\documentclass[preprint,11pt,nopreprintline]{elsarticle}
\usepackage[letterpaper,margin=1in]{geometry}
\usepackage[T1]{fontenc}
\usepackage{lmodern,microtype,mathtools,amssymb,amsthm,mathrsfs,booktabs,array,enumitem}
\usepackage{tikz}
\usepackage{float,etoolbox,needspace}
\usepackage[hidelinks]{hyperref}
\hypersetup{pdftitle={Ooms spectra of Frobenius maximal parabolics: strict unimodality and Euclidean log-concavity},pdfauthor={Vincent E. Coll, Jr.}}
\numberwithin{equation}{section}
\newtheorem{theorem}{Theorem}[section]
\newtheorem{proposition}[theorem]{Proposition}
\newtheorem{lemma}[theorem]{Lemma}
\newtheorem{corollary}[theorem]{Corollary}
\theoremstyle{definition}

\newtheorem{conjecture}[theorem]{Conjecture}
\theoremstyle{remark}
\newtheorem{remark}[theorem]{Remark}
\newtheorem{example}[theorem]{Example}
\newcommand{\rev}{\vee}
\newcommand{\cut}[2]{\left[#1\right]_{\le #2}}
\newcommand{\PA}{\mathcal A}
\newcommand{\PC}{\mathcal C}
\newcommand{\supp}{\operatorname{supp}}
\newcommand{\ord}{\operatorname{ord}}
\newcommand{\Z}{\mathbb Z}

\newcommand{\C}{\mathbb C}

\newcommand{\CodeRepository}{}
\newcommand{\ad}{\operatorname{ad}}
\begin{document}
\begin{frontmatter}
\title{Ooms spectra of Frobenius maximal parabolics: strict unimodality and Euclidean log-concavity}
\author{Vincent E. Coll, Jr.\corref{cor1}\fnref{orcid}}
\address{Department of Mathematics, Lehigh University, Bethlehem, Pennsylvania 18015, USA}
\ead{vec208@lehigh.edu}
\cortext[cor1]{Corresponding author.}
\fntext[orcid]{ORCID: 0000-0002-5775-5522.}
\begin{abstract}
We study the Ooms multiplicities of Frobenius maximal parabolics $W(a,b)$ of type~$A$.
For $a=mb+r$ with $1\le r<b$, the Euclidean algorithm relates their potential histograms to those of the smaller algebra $W(r,b-r)$.
We prove strict unimodality for every positive coprime pair: the multiplicities increase strictly to the equal central values at eigenvalues $0$ and $1$, then decrease strictly.
This strengthens the recent maximal-parabolic unimodality theorem of Giaquinto, Irving, Lauve, and Mastnak.
We also prove log-concavity for the unbounded families $W(mb\pm t,b)$, where $t\in\{1,2,3\}$, $b>t$, $\gcd(b,t)=1$, and $m\ge1$.
These results cover every Frobenius case with smaller block at most $8$, and also smaller block $10$.
For each fixed residue class, we then prove that log-concavity of the whole family is determined by finitely many initial spectra.
The finite cutoff is obtained symbolically. Exact integer verification of the resulting finite lists proves strict internal log-concavity whenever the smaller block is at most $128$, with no bound on the larger block.
The unrestricted log-concavity conjecture is not proved here.
\end{abstract}
\end{frontmatter}

\section{Introduction}
A finite-dimensional complex Lie algebra $\mathfrak g$ is \emph{Frobenius} if there is a linear functional $\varphi\in\mathfrak g^*$ for which the bilinear form
\[
 B_\varphi(x,y)=\varphi([x,y])
\]
is nondegenerate. Such a functional determines a unique \emph{principal element} $h$ by
\[
 \varphi([h,x])=\varphi(x)\qquad(x\in\mathfrak g).
\]
For the Frobenius algebraic Lie algebras considered here, the multiset of eigenvalues of $\ad h$, including algebraic multiplicities, is independent of the Frobenius functional; see Ooms~\cite{Ooms} and Gerstenhaber--Giaquinto~\cite[Theorem~3]{GG}. We call this invariant multiset the \emph{Ooms spectrum}. Unimodality and log-concavity of an Ooms spectrum will always refer to its multiplicities, ordered by increasing distinct eigenvalue, rather than to the eigenvalues themselves. For Frobenius seaweeds in type~$A$, the distinct eigenvalues form an unbroken interval of integers and the multiplicities are symmetric about $1/2$~\cite{CHM}.

We study the type-$A$ maximal parabolic
\begin{equation}\label{eq:W}
 W(a,b)=\left\{\begin{pmatrix}X&Z\\0&Y\end{pmatrix}:
 X\in\operatorname{Mat}_a(\C),\ Y\in\operatorname{Mat}_b(\C),\
 Z\in\operatorname{Mat}_{a\times b}(\C),\ \operatorname{tr}X+\operatorname{tr}Y=0\right\}.
\end{equation}
It is Frobenius exactly when $\gcd(a,b)=1$. This criterion is a special case of the meander index formula of Dergachev and Kirillov~\cite{DK}; see also~\cite[Section~3]{MR}. Throughout the paper, block sizes are positive and coprime.

A positive sequence $(c_0,\ldots,c_t)$ is \emph{log-concave} if $c_j^2\ge c_{j-1}c_{j+1}$ for $1\le j<t$, and \emph{strictly internally log-concave} if all these inequalities are strict. A two-term sequence has no internal inequality. We use \emph{strictly unimodal with a central pair} to mean strictly increasing up to two equal middle terms and strictly decreasing thereafter. This allows the two-term sequence $(1,1)$.

Coll, Magnant, and Wang~\cite[Conjecture~21]{CMW} conjectured that the multiplicities of every Frobenius type-$A$ seaweed are strictly unimodal. Mayers and Russoniello~\cite{MR} subsequently proved the unimodal spectrum property for several families and, in a number of maximal-parabolic families, obtained the stronger property of log-concavity. The present paper addresses the maximal-parabolic class; it makes no assertion about strict unimodality for arbitrary Frobenius seaweeds. The unrestricted maximal-parabolic log-concavity conjecture remains open.

Giaquinto, Irving, Lauve, and Mastnak~\cite[Corollary~15 and Theorem~16]{GILM} have recently proved unimodality for every Frobenius maximal parabolic of type~$A$, using fence posets and good gradings. Their preprint appeared while the present manuscript was being prepared for posting, and we make no priority claim for the unimodality statement itself. Their conclusion permits equal adjacent multiplicities away from the central pair; it does not assert strict unimodality. They also prove a unimodality theorem for the adjoint action on the ambient $\mathfrak{gl}_{a+b}$, which is not addressed here. The Euclidean argument developed here is different: Theorem~\ref{thm:mainunimodal} strengthens the maximal-parabolic conclusion to strict unimodality, and the later results give additional log-concavity statements.

Our main additional ingredient is a subtraction-free recursion for spectral first differences. The underlying histogram recursions also occur in~\cite{GILM}; see Section~\ref{sec:geometry}.
When using residues, we order the blocks so that $a\ge b\ge2$ and write $a=mb+r$, with $m\ge1$ and $1\le r<b$.
The smaller algebra $W(r,b-r)$ then determines the histograms of the entire family with that residue.
Block size one is treated separately.

\begin{theorem}[Strict unimodality]\label{thm:mainunimodal}
For all positive coprime $a,b$, the multiplicity sequence of the Ooms spectrum of $W(a,b)$ is strictly unimodal with a central pair. The two central eigenvalues are $0$ and $1$.
\end{theorem}

In particular, Theorem~\ref{thm:mainunimodal} proves the strict-unimodality conjecture of~\cite{CMW} for every Frobenius maximal parabolic.

The proof is algebraic after the directed-meander identification of the spectrum. The Euclidean moves
\[
 L(a,b)=(a+b,b),\qquad R(a,b)=(a,a+b)
\]
generate every positive coprime pair uniquely from $(1,1)$. We prove corresponding histogram recursions and retain four truncated profiles, one of which records the rising-half first differences of the spectrum. A compatibility identity converts their apparent subtraction into a positive boundary term. The resulting recursion proves Theorem~\ref{thm:mainunimodal} by induction, without using a general fence-rank unimodality theorem.

The same construction gives log-concavity for the following residue families.

\Needspace{8\baselineskip}
\begin{theorem}[Residues $\pm1$, $\pm2$, and $\pm3$]\label{thm:residues}
Let $m\ge1$.
\begin{enumerate}[label=\textup{(\roman*)},nosep]
\item For every $b\ge2$, both $W(mb+1,b)$ and $W(mb-1,b)$ are log-concave whenever the first block is positive.
\item If $b\ge3$ is odd, then both $W(mb+2,b)$ and $W(mb-2,b)$ are log-concave.
\item If $b\ge4$ and $3\nmid b$, then both $W(mb+3,b)$ and $W(mb-3,b)$ are log-concave.
\end{enumerate}
There is no upper bound on $b$ or $m$ in any of these families.
\end{theorem}

Consequently, after ordering the blocks so that $a\ge b\ge2$, the congruence
$a\equiv\pm1,\pm2,$ or $\pm3\pmod b$, with a residue coprime to $b$, implies log-concavity.
For $b=1$, the multiplicities are $(1,2,\ldots,a,a,\ldots,2,1)$ and are log-concave directly.
Parts~A--C of the mathematical supplement~\cite{SR} give the full symbolic proofs of Theorem~\ref{thm:residues}.

\medskip\noindent\textbf{Examples.}
The theorem can be applied directly from a signed remainder representation of $a$ modulo $b$:
\[
\begin{array}{c|c|c}
W(a,b) & a=mb+r & \text{part of Theorem~\ref{thm:residues}}\\ \hline
W(13,4) & 13=3\cdot4+1 & r=+1\\
W(11,4) & 11=3\cdot4-1 & r=-1\\
W(17,5) & 17=3\cdot5+2 & r=+2\\
W(13,5) & 13=3\cdot5-2 & r=-2\\
W(17,7) & 17=2\cdot7+3 & r=+3\\
W(11,7) & 11=2\cdot7-3 & r=-3
\end{array}
\]
Thus all six spectra are log-concave without computing their multiplicities first.

\begin{corollary}\label{cor:smallblocks}
Every Frobenius maximal parabolic whose smaller block is at most $8$ is log-concave.  The same is true when the smaller block is $10$.
\end{corollary}
\begin{proof}
For $b=3,4,5,6,7,8,10$, every unit modulo $b$ is represented by one of $\pm1,\pm2,\pm3$.  The cases $b=1,2$ are immediate (and also occur in the earlier formulas of Mayers--Russoniello).  Apply Theorem~\ref{thm:residues}, interchanging the two blocks when necessary.

The omission of $9$ is genuine and explains the shape of the statement.  Modulo $9$ the units are
\[
\{1,2,4,5,7,8\}=\{\pm1,\pm2,\pm4\},
\]
so the two residue classes $\pm4$ are not covered by Theorem~\ref{thm:residues}.  By contrast, modulo $10$ the units are
\[
\{1,3,7,9\}=\{\pm1,\pm3\},
\]
and hence every Frobenius $W(a,10)$ is covered.  The omission of $9$ is only an omission from this purely symbolic consequence of Theorem~\ref{thm:residues}; it is not a counterexample.  The exact finite verification in Corollary~\ref{thm:main128} below also proves strict internal log-concavity for every Frobenius $W(a,9)$.
\end{proof}

Together with the elementary block-size-one case, Theorem~\ref{thm:residues} contains the maximal-parabolic log-concavity conclusions of Mayers and Russoniello~\cite[Corollaries~35, 39, 43, and~46 in the arXiv version]{MR}: their fixed-block-$2$ family is residue $\pm1$, their adjacent-block family is residue $+1$, and their difference-two family is residue $+2$ with odd modulus. Their explicit spectrum formulas and their non-maximal seaweed families remain separate results.

\Needspace{6\baselineskip}
\medskip\noindent\textbf{Concrete example: all $W(a,3)$.}
If $\gcd(a,3)=1$, then $a\equiv\pm1\pmod3$, so Theorem~\ref{thm:residues}(i) proves log-concavity for every Frobenius $W(a,3)$.  The first cases make the shape visible.  The table records only the rising half; appending its reversal gives the full sequence.
The least defect is the minimum of $c_j^2-c_{j-1}c_{j+1}$ over the internal indices of the full sequence.
\[
\begin{array}{c|c|l|c}
 a & a\bmod3 & \text{rising half of the multiplicity sequence} & \text{least defect}\\ \hline
4&1&(1,3,6,8)&3\\
5&2&(1,4,8,11)&8\\
7&1&(1,3,7,12,16)&2\\
8&2&(1,4,9,15,19)&7\\
10&1&(1,3,7,13,20,25)&2\\
11&2&(1,4,9,16,23,28)&7
\end{array}
\]
For example, $W(5,3)$ has full multiplicity sequence
\[
(1,4,8,11,11,8,4,1).
\]

The next two block sizes are equally concrete.  For $b=4$ the Frobenius condition forces $a\equiv\pm1\pmod4$, so part~(i) applies to every $W(a,4)$:
\[
\begin{array}{c|c|l}
 a & a\bmod4 & \text{rising half of the multiplicity sequence}\\ \hline
5&1&(1,3,6,9,11)\\
7&3&(1,4,9,14,18)\\
9&1&(1,3,7,13,19,23)\\
11&3&(1,4,10,18,26,31)
\end{array}
\]
For $b=5$, every nonzero residue is one of $\pm1,\pm2$, so parts~(i)--(ii) cover every Frobenius $W(a,5)$:
\[
\begin{array}{c|c|l}
 a & a\bmod5 & \text{rising half of the multiplicity sequence}\\ \hline
6&1&(1,3,6,9,12,14)\\
7&2&(1,4,10,17,22)\\
8&3&(1,5,12,20,26)\\
9&4&(1,4,9,15,21,25)
\end{array}
\]
These tables use the histogram construction of Section~\ref{sec:geometry}.

For larger blocks, $W(1001,100)$ is residue $+1$, $W(1003,143)$ is residue $+2$, and $W(1003,125)$ is residue $+3$; all three are log-concave by Theorem~\ref{thm:residues} without any finite-size restriction.

The same idea extends to every fixed residue class. Fix $b\ge2$ and $1\le s<b$ with $\gcd(s,b)=1$, and consider
\[
 W(s,b),\ W(s+b,b),\ W(s+2b,b),\ldots,
\]
or equivalently the family $W(s+qb,b)$ with $q\ge0$.  Let
\[
 \ell=\max_i\psi(i)-\min_i\psi(i)
\]
for the initial pair $(s,b)$.  Equivalently, after a common monomial shift, the two histograms introduced in Section~\ref{sec:geometry} have support in $[0,\ell]$.

\begin{theorem}[Finite verification for a fixed residue class]\label{thm:mainray}
For the family $W(s+qb,b)$, $q\ge0$, all spectra are log-concave if and only if the spectra for
\[
 q=0,1,\ldots,2\ell+4
\]
are log-concave. The same equivalence holds for strict internal log-concavity.
For $b=1$, the assertion holds with initial pair $(1,1)$ and $\ell=1$.
\end{theorem}

Thus each fixed $(s,b)$ requires only $2\ell+5$ initial checks.

\begin{example}[The family $W(2q+1,2)$]
Starting from $(1,2)$ and repeatedly adding the second block to the first gives
\[
 (1,2)\longrightarrow(3,2)\longrightarrow(5,2)\longrightarrow(7,2)\longrightarrow\cdots.
\]
The initial histograms are $A=z^2$ and $C=1+z$, so $\ell=2$. The first three multiplicity sequences are
\[
\begin{array}{c|c}
(a,b)&\text{multiplicities}\\ \hline
(1,2)&(1,2,2,1)\\
(3,2)&(1,3,5,5,3,1)\\
(5,2)&(1,3,6,9,9,6,3,1).
\end{array}
\]
Theorem~\ref{thm:mainray} says that it is enough to check $q=0,\ldots,8$; those nine cases determine log-concavity for every $W(2q+1,2)$.
\end{example}

A less elementary example used in the computation is the family $W(4+9q,9)$.  For its initial pair $(4,9)$, one has $\ell=6$, so it is enough to check $q=0,\ldots,16$.  This finite reduction is the mechanism behind the bound $128$ below.

\begin{corollary}[Exact verification through smaller block $128$]\label{thm:main128}
If $\gcd(a,b)=1$ and the smaller of the two blocks $a,b$ has size at most $128$, then the Ooms spectrum of $W(a,b)$ is strictly internally log-concave. There is no bound on the larger block.
\end{corollary}

The finite comparisons required for Corollary~\ref{thm:main128} were performed by the supplied exact-arithmetic programs.
The number $128$ records the completed range, not a mathematical threshold.
Section~\ref{sec:certificate} gives the deduction from Theorem~\ref{thm:mainray}; implementation details are in the supplementary files.

It is useful to keep the two log-concavity results separate.  Corollary~\ref{thm:main128} is broader for bounded smaller block: it subsumes Corollary~\ref{cor:smallblocks}, all of the displayed $b=3,4,5$ examples, and in particular every Frobenius case with smaller block $9$.  Theorem~\ref{thm:residues}, however, is stronger in a different direction: its residue families are proved symbolically for arbitrarily large $b$ and arbitrarily large Euclidean quotient.  Thus the two results overlap, but neither should be viewed as replacing the other.

Section~\ref{sec:geometry} develops the histogram model and Euclidean lifting. Section~\ref{sec:residues}, with the supplement, proves the residue-family theorem. Sections~\ref{sec:kernel} and~\ref{sec:positive} prove strict unimodality. Section~\ref{sec:rays} establishes the finite-ray theorem, which is applied in Section~\ref{sec:certificate} to smaller blocks through $128$.

\section{The Euclidean algorithm and histogram lifting}\label{sec:geometry}
\subsection{The spectral formula}
Put $v=a+b$. Place vertices $1,\ldots,v$ in order on a line. The top edges join opposite vertices within the intervals $[1,a]$ and $[a+1,v]$; the bottom edges join opposite vertices within $[1,v]$. Fixed vertices of a reflection carry no edge. Orient each top edge from right to left and each bottom edge from left to right. For coprime $a,b$, this meander is a single path~\cite{DK,MR}.

The integer potential $\psi$ is specified, up to a common additive constant, by
\begin{equation}\label{eq:potential}
 \psi(y)=\psi(x)+1\ \text{on a top edge }x<y,
 \qquad
 \psi(x)=\psi(y)+1\ \text{on a bottom edge }x<y.
\end{equation}
We initially normalize $\psi(v)=0$. The standard meander construction identifies the diagonal of a principal element with these values, after subtracting their average. Consequently, the eigenvalue at an allowed matrix unit $e_{ij}$ is $\psi(i)-\psi(j)$~\cite{GG,CD,MR}. The $v-1$ traceless diagonal directions have eigenvalue zero.

For a Laurent polynomial $P$, write $P^\rev(z)=P(z^{-1})$. Define the block histograms
\[
 A(z)=\sum_{i=1}^{a}z^{\psi(i)},\qquad
 C(z)=\sum_{i=a+1}^{v}z^{\psi(i)}.
\]
We work with the reversed Ooms multiplicity polynomial $M$, whose coefficient at exponent $k$ is the multiplicity of the actual eigenvalue $-k$.

\begin{proposition}[Histogram formula]\label{prop:spectrum}
The complete reversed Ooms multiplicity polynomial is
\begin{equation}\label{eq:M}
 M(A,C)=A^\rev A+A^\rev C+C^\rev C-1.
\end{equation}
A common monomial shift of $A,C$ does not change $M$. Its total coefficient sum is
\begin{equation}\label{eq:dim}
 M(1)=a^2+ab+b^2-1.
\end{equation}
\end{proposition}
\begin{proof}
The three products in~\eqref{eq:M} count the differences $\psi(j)-\psi(i)$ in the three allowed matrix blocks of~\eqref{eq:W}. These are the negatives of the corresponding eigenvalues. The products count all $v$ diagonal matrix units at zero; the traceless condition removes one of them. This explains the scalar $-1$. The common-shift assertion and~\eqref{eq:dim} follow directly.
\end{proof}

\subsection{The lifting step}
Let $1\le s<b$ with $\gcd(s,b)=1$, and let $u_1,\ldots,u_b$, with $u_b=0$, be the potential of $W(s,b-s)$. Denote its two histograms by $P,Q$ and put
\[
 I_j=1+z+\cdots+z^{j-1}\quad(j\ge1),\qquad I_0=0.
\]

\begin{proposition}[Euclidean histogram lifting]\label{prop:lift}
For every $m\ge0$, the histograms of $W(mb+s,b)$, normalized at the last vertex, are
\begin{equation}\label{eq:lift}
 H_m=P I_{m+1}+Q I_m,\qquad G=z^{-1}(P+Q).
\end{equation}
In particular, $m=0$ describes the actual meander of $W(s,b)$, and
\begin{equation}\label{eq:Hstep}
 H_{m+1}=zH_m+(P+Q)=z(H_m+G).
\end{equation}
\end{proposition}
\begin{proof}
The seed top reflection is
\[
 \tau(i)=\begin{cases}s+1-i,&i\le s,\\b+s+1-i,&i>s.\end{cases}
\]
For $m\ge1$, set $q_i=m$ if $i\le s$ and $q_i=m-1$ otherwise. The parent potential is
\begin{align}
 \psi(i+jb)&=\min\{u_i+2j,\ u_{\tau(i)}+1+2(q_i-j)\},
       &&0\le j\le q_i,\label{eq:minlift}\\
 \psi(mb+s+t)&=u_{b+1-t}-1,&&1\le t\le b.\label{eq:smalllift}
\end{align}
\ref{app:lift} verifies every edge relation and counts the paired residue chains. For a two-element seed orbit $i<\tau(i)$, its combined histogram is $z^{u_i}(1+z)I_{q_i+1}$; for a fixed point it is $z^{u_i}I_{q_i+1}$. Summing over the two seed blocks gives~\eqref{eq:lift} for $m\ge1$.

At $m=0$, assign the values $u_i$ to $1\le i\le s$ and the values $u_{b+1-t}-1$ to the second block. First-block top edges retain their seed relations. Second-block top edges are reversed seed bottom edges. Bottom edges incident with the first block have values $u_i,u_i-1$. A bottom edge contained in the second block becomes a reversed top edge in the second seed block. Thus all edge relations hold. Since $u_1=1$, the last value is zero. This proves the actual $m=0$ assertion. Equation~\eqref{eq:Hstep} follows from $I_{j+1}=1+zI_j$.
\end{proof}

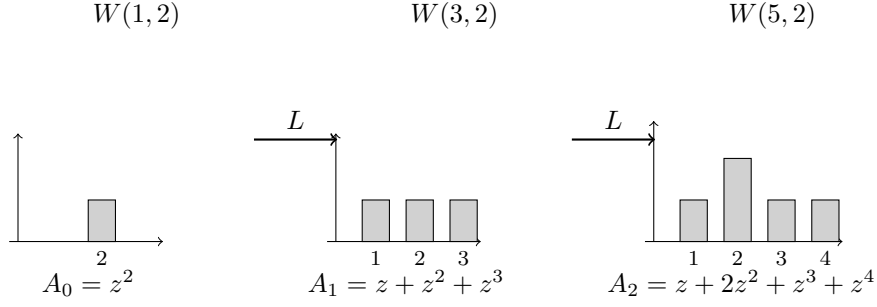
\begin{figure}[H]
\centering
\begin{tikzpicture}[font=\small]
  \node at (1.8,3.0) {$W(1,2)$};
  \node at (6.0,3.0) {$W(3,2)$};
  \node at (10.2,3.0) {$W(5,2)$};
  \begin{scope}[x=0.58cm,y=0.55cm]
    \draw[->] (0.2,0)--(3.7,0);
    \draw[->] (0.4,0)--(0.4,2.6);
    \foreach \x/\h in {2/1}{\draw[fill=black!18] (\x,0) rectangle ++(0.62,\h);}
    \node[below] at (2.31,0) {\scriptsize 2};
    \node at (2.0,-1.0) {$A_0=z^2$};
  \end{scope}
  \begin{scope}[xshift=4.2cm,x=0.58cm,y=0.55cm]
    \draw[->] (0.2,0)--(3.7,0);
    \draw[->] (0.4,0)--(0.4,2.6);
    \foreach \x/\h in {1/1,2/1,3/1}{\draw[fill=black!18] (\x,0) rectangle ++(0.62,\h);}
    \foreach \x in {1,2,3}{\node[below] at (\x+0.31,0) {\scriptsize \x};}
    \node at (2.0,-1.0) {$A_1=z+z^2+z^3$};
  \end{scope}
  \begin{scope}[xshift=8.4cm,x=0.58cm,y=0.55cm]
    \draw[->] (0.2,0)--(4.7,0);
    \draw[->] (0.4,0)--(0.4,2.9);
    \foreach \x/\h in {1/1,2/2,3/1,4/1}{\draw[fill=black!18] (\x,0) rectangle ++(0.62,\h);}
    \foreach \x in {1,2,3,4}{\node[below] at (\x+0.31,0) {\scriptsize \x};}
    \node at (2.4,-1.0) {$A_2=z+2z^2+z^3+z^4$};
  \end{scope}
  \draw[->,thick] (3.35,1.35)--(4.45,1.35) node[midway,above] {$L$};
  \draw[->,thick] (7.55,1.35)--(8.65,1.35) node[midway,above] {$L$};
  \node[align=center] at (6.2,-2.05) {The second-block histogram stays fixed: $C=1+z$.\\Each step uses $A_{q+1}=z(A_q+C)$.};
\end{tikzpicture}
\caption{Euclidean enlargement in the fixed residue class $W(1+2q,2)$.  The figure shows the first-block histogram for $q=0,1,2$; the second-block histogram is unchanged.}
\label{fig:lifting12}
\end{figure}

\begin{corollary}[Euclidean histogram moves]\label{cor:moves}
Up to a common monomial normalization, the block operations $L(a,b)=(a+b,b)$ and $R(a,b)=(a,a+b)$ act on histograms by
\begin{equation}\label{eq:moves}
 L(A,C)=(z(A+C),C),\qquad R(A,C)=(zA,A+C).
\end{equation}
Starting at $(A,C)=(z,1)$, these rules give the histograms of every positive coprime pair.
\end{corollary}
\begin{proof}
For $b>1$, write $a=mb+s$ and apply~\eqref{eq:Hstep}; this gives the left rule. For $b=1$, the histograms are $z+\cdots+z^a$ and $1$, giving the same rule directly. Reversing the vertices and replacing $\psi$ by $1-\psi$ transforms the histograms for $(a,b)$ into $(zC^\rev,zA^\rev)$ for $(b,a)$. Apply the left rule to the swapped pair and swap back. The result is $(A,z^{-1}(A+C))$; multiply both histograms by $z$ to obtain the right rule in~\eqref{eq:moves}.

Every positive coprime pair other than $(1,1)$ has a unique predecessor obtained by subtracting the smaller entry from the larger. The sum decreases, so the descent ends at $(1,1)$. Reversing it proves the final assertion and uniqueness of the word.
\end{proof}

The word in $L,R$ is the subtractive Euclidean algorithm read backwards.

The same histogram recursions appear in~\cite[equation~(3)]{GILM}, with their letters $L,R$ interchanged relative to ours. Proposition~\ref{prop:lift} supplies the residue-chain derivation used here.

\begin{example}\label{ex:53}
The pair $(5,3)$ has Euclidean word $LRL$:
\[
(1,1)\xrightarrow{L}(2,1)\xrightarrow{R}(2,3)\xrightarrow{L}(5,3).
\]
The corresponding histograms are
\[
 A=z+z^2+2z^3+z^4,\qquad C=1+z+z^2.
\]
Its Ooms multiplicity sequence is
\[
 (1,4,8,11,11,8,4,1),
\]
indexed by the actual eigenvalues $-3,-2,\ldots,4$. Its sum is $48=5^2+5\cdot3+3^2-1$.
\end{example}

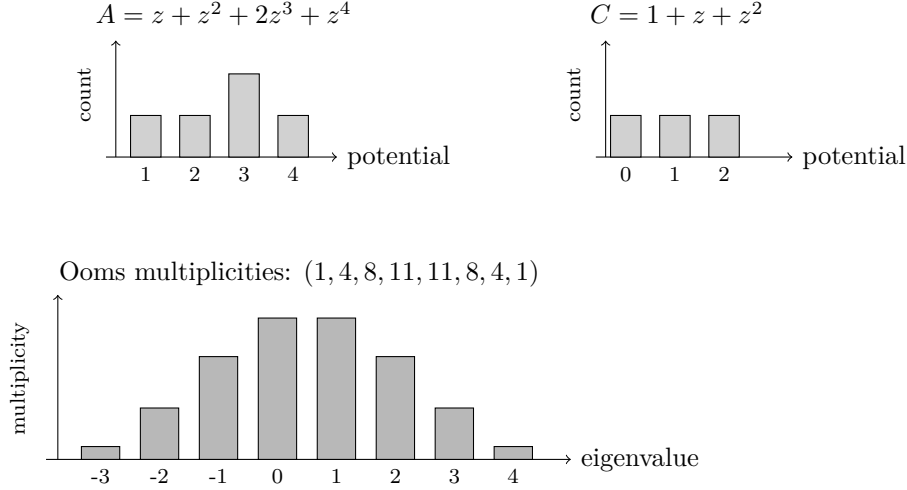
\begin{figure}[H]
\centering
\begin{tikzpicture}[font=\small]
  % A histogram
  \begin{scope}[x=0.65cm,y=0.55cm]
    \draw[->] (0.5,0)--(5.2,0) node[right] {potential};
    \draw[->] (0.7,0)--(0.7,2.8);\node[rotate=90] at (0.05,1.5) {\scriptsize count};
    \foreach \x/\h in {1/1,2/1,3/2,4/1}{
      \draw[fill=black!18] (\x,0) rectangle ++(0.62,\h);
      \node[below] at (\x+0.31,0) {\scriptsize \x};
    }
    \node at (2.9,3.45) {$A=z+z^2+2z^3+z^4$};
  \end{scope}
  % C histogram
  \begin{scope}[xshift=7.0cm,x=0.65cm,y=0.55cm]
    \draw[->] (-0.3,0)--(3.7,0) node[right] {potential};
    \draw[->] (-0.1,0)--(-0.1,2.8);\node[rotate=90] at (-0.75,1.5) {\scriptsize count};
    \foreach \x/\h in {0/1,1/1,2/1}{
      \draw[fill=black!18] (\x,0) rectangle ++(0.62,\h);
      \node[below] at (\x+0.31,0) {\scriptsize \x};
    }
    \node at (1.35,3.45) {$C=1+z+z^2$};
  \end{scope}
  % Spectrum histogram
  \begin{scope}[yshift=-4.0cm,x=0.78cm,y=0.17cm]
    \draw[->] (-0.6,0)--(8.3,0) node[right] {eigenvalue};
    \draw[->] (-0.4,0)--(-0.4,12.8);\node[rotate=90] at (-1.05,6.3) {\scriptsize multiplicity};
    \foreach \x/\h/\lab in {0/1/-3,1/4/-2,2/8/-1,3/11/0,4/11/1,5/8/2,6/4/3,7/1/4}{
      \draw[fill=black!28] (\x,0) rectangle ++(0.65,\h);
      \node[below] at (\x+0.325,0) {\scriptsize \lab};
    }
    \node at (3.7,14.5) {Ooms multiplicities: $(1,4,8,11,11,8,4,1)$};
  \end{scope}
\end{tikzpicture}
\caption{The two block histograms for $W(5,3)$ and the resulting Ooms multiplicities.  Formula~\eqref{eq:M} turns differences of potential heights in the two blocks into the spectral multiplicities.}
\label{fig:hist53}
\end{figure}

\section{Small residues and symbolic log-concavity}\label{sec:residues}
We prove Theorem~\ref{thm:residues} from the lifted histograms. The full coefficient calculations are in Parts~A--C of the mathematical supplement~\cite{SR}, which treat residues $\pm1$, $\pm2$, and $\pm3$, respectively, and form part of the proof.

Put $I_j=1+z+\cdots+z^{j-1}$ for $j\ge1$ and $I_0=0$.  Proposition~\ref{prop:lift} is the only geometric input: once the seed is known, the parent histograms are explicit in the Euclidean quotient.

\subsection{Residue plus or minus 1}
For $W(mb\pm1,b)$ the Euclidean seed has a block of size one: it is $W(1,b-1)$ or $W(b-1,1)$. The resulting spectrum has the factorizations below.  If $M^+_{b,m}$ and $M^-_{b,m}$ denote the reversed multiplicity polynomials, then
\begin{align}
 z^{b+m}M^+_{b,m}
 &=I_b\bigl(I_{b+2m+1}+zI_bI_mI_{m+1}\bigr),\label{eq:r1plus}\\
 z^{b+m-1}M^-_{b,m}
 &=I_b\bigl(I_{b+2m-1}+zI_bI_{m-1}I_m+zI_{b-2}I_{2m}\bigr).\label{eq:r1minus}
\end{align}
The factorizations follow from Proposition~\ref{prop:lift}; their algebraic derivation is given in~\cite[Proposition~A.4.1]{SR}.
The following special inequality explains why the first factorization gives log-concavity.

\begin{lemma}[An interval-polynomial inequality]\label{lem:intervalsum}
For positive integers $p,q,r$, the coefficients of
\[
 S(z)=I_{p+q+r}+zI_pI_qI_r
\]
are symmetric, positive on their support, and log-concave.
\end{lemma}
\begin{proof}
Permute the lengths so that $p\le q\le r$ and put $L=p+q+r$.
By symmetry it suffices to check $1\le j\le\lfloor(L-1)/2\rfloor$.
For $1\le j<p$, the three relevant coefficients are given by
$f(i)=1+i(i+1)/2$, and
\[
 f(j)^2-f(j-1)f(j+1)=\frac{(j-1)(j+2)}2\ge0.
\]
For $p\le j\le\lfloor(L-1)/2\rfloor$, coefficient extraction gives
\[
 S_{j+1}-2S_j+S_{j-1}
 =1-\mathbf1_{j\ge p}-\mathbf1_{j\ge q}-\mathbf1_{j\ge r}
       +\mathbf1_{j\ge p+q}\le0.
\]
The other pair terms are absent because $j<p+r$ and $j<q+r$.
If $j<p+q$, the first subtraction suffices; if $j\ge p+q$, the
$p$ and $q$ terms cancel the constant and pair term.
Ordinary concavity of a nonnegative triple implies log-concavity.
\end{proof}

Apply Lemma~\ref{lem:intervalsum} to the bracket in~\eqref{eq:r1plus}, with lengths $b,m,m+1$.
For~\eqref{eq:r1minus}, the required additional inequality is
\cite[Lemma~A.3.3]{SR}; its proof treats $m=1$ separately and computes the three possible initial defects before ordinary concavity begins.
Convolution with $I_b$ preserves log-concavity for nonnegative sequences without internal zeros, as proved in~\cite[Lemma~A.3.1]{SR}.
This proves Theorem~\ref{thm:residues}(i).

\subsection{Residues plus or minus 2 and plus or minus 3}
For the next two residues, explicit coefficient formulas give a finite initial segment followed by a region of nonincreasing second differences.
The following elementary criterion is the step that extends the initial inequalities to the whole spectrum.

\begin{lemma}[From the initial segment to the tail]\label{lem:headtail}
Let $(L_0,\ldots,L_N,L_N,\ldots,L_0)$ be a nonnegative symmetric sequence, and put
\[
 d_k=L_k-L_{k-1},\qquad w_k=L_{k+1}-2L_k+L_{k-1},\qquad
 \Delta_k=L_k^2-L_{k-1}L_{k+1}.
\]
For $1\le K\le N-1$, suppose that $\Delta_k\ge0$ for $1\le k\le K$,
that $d_{K+1},d_N\ge0$, and that
$w_{k+1}\le w_k$ for $K\le k\le N-2$.
Then the full symmetric sequence is log-concave.
\end{lemma}
\begin{proof}
The sequence $d_{K+1},\ldots,d_N$ is ordinarily concave and has nonnegative endpoints, so all its terms are nonnegative.
Direct expansion gives
\[
 \Delta_{k+1}-\Delta_k
 =w_kd_{k+1}+L_k(w_k-w_{k+1}).
\]
As long as $w_k\ge0$, the defects cannot decrease.
Once $w_k<0$, every later second difference is nonpositive, and ordinary concavity gives the remaining inequalities.
The central inequalities follow from $d_N\ge0$; reflection gives the other half.
When $K=N-1$, only this central check remains.
\end{proof}

For residue $\pm2$, write $b=2n+1$.  In a common normalization the two seed states are
\[
 A_+=z^nI_2,\qquad C_+=1+zI_2I_{n-1},
\]
\[
 A_-=z^{n+1}+zI_2I_{n-1},\qquad C_-=I_2.
\]
Substitution into~\eqref{eq:M} and Proposition~\ref{prop:lift} gives the coefficient formula of~\cite[Proposition~B.2.2]{SR}.
Write $p=\min(n,m)$. For the plus family take $K=p+1$, and for the minus family take $K=p$.
The final slope is nonnegative by~\cite[Lemma~B.3.1]{SR}, and the second differences are nonincreasing after $K$ by~\cite[Proposition~B.4.2]{SR}.
For $p\ge3$, five parameter cases, with two transition indices for each sign, give $20$ quartics in $u=p-3$; their coefficients are listed in~\cite[Section~B.5.1]{SR}.
For example, in the plus family with $n=m=p$,
\[
 3\Delta_p=4u^4+46u^3+173u^2+323u+357>0\qquad(u\ge0).
\]
The same section proves positivity of the free initial segment and the first tail slope.
Section~B.5.2 of the supplement treats $p=1,2$ by eight fixed initial patterns; Section~B.6 lists the sixteen remaining finite base cases.
These cases exhaust $n,m\ge1$, so Lemma~\ref{lem:headtail} proves Theorem~\ref{thm:residues}(ii).

For residue $\pm3$, write $b=3n+r$, $r\in\{1,2\}$.  The plus seed has
\[
 A_+=z^nI_3,\qquad C_+=I_r+zI_3I_{n-1},
\]
and the minus seed has $C_-=I_3$ and
\[
 A_-=
 \begin{cases}
 z^{n+2}+zI_3I_{n-1},&r=1,\\
 z^{n+1}I_2+zI_3I_{n-1},&r=2.
 \end{cases}
\]
The four rational identities and the resulting coefficient formula are
\cite[Proposition~C.2.2 and Corollary~C.2.3]{SR}.
For $p=\min(n,m)\ge6$, take $K=p+2$ for plus and $K=p+1$ for minus.
There are seven parameter cases for each of four sign/residue choices, and three transition indices in each case.
Section~C.8 of the supplement lists all $84$ quartics in $u=p-6$, each with positive coefficients.
For $p\le5$ and $\max(n,m)\ge p+4$, the initial segment is independent of the larger parameter; the forty resulting cases are in Section~C.5.2, with $K=p+3$ for plus and $K=p+2$ for minus.
Section~C.9 covers the $140$ remaining finite cases.
The final slope and the decreasing second differences are established in~\cite[Lemma~C.3.1 and Proposition~C.4.2]{SR}.
The positive initial defects and slopes in Section~C.5 therefore meet the hypotheses of Lemma~\ref{lem:headtail}.
This proves Theorem~\ref{thm:residues}(iii).

\subsection{Closed formulas for the first nontrivial modulus}
For $b=3$ one can see the Euclidean structure directly.  Write $a=3m+1$ or $a=3m+2$.  For $m\ge5$, the rising half $L_0,\ldots,L_{m+2}$ is
\[
L_j=
\begin{cases}
j^2+j+1,&0\le j\le4,\\
9j-15,&5\le j\le m,\\
9m-7,&j=m+1,\\
9m-2,&j=m+2,
\end{cases}
\qquad(a=3m+1),
\]
and
\[
L_j=
\begin{cases}
(j+1)^2,&0\le j\le3,\\
9j-12,&4\le j\le m,\\
9m-4,&j=m+1,\\
9m+1,&j=m+2,
\end{cases}
\qquad(a=3m+2).
\]
The arithmetic core has log-concavity defect $81$; the finitely many head and tail defects are positive.  The cases $1\le m\le4$ follow by the same histogram formula; the introduction displays $m=1,2,3$. The cases $a=1,2$ follow from the elementary small-block cases.

\section{Algebraic symmetry and block interchange}\label{sec:kernel}
All subsequent identities use the canonical pair generated by~\eqref{eq:moves}. If the word has length $N$, put $n=N+1$. Induction shows that $A$ is monic of degree $n$, with zero constant coefficient, and that $C$ has constant coefficient one and degree at most $n-1$.

\begin{lemma}[Symmetry and block interchange]\label{lem:symmetry}
Every generated state satisfies
\begin{equation}\label{eq:conserved}
 AC^\rev-zA^\rev C=(z-1)(A^\rev A+C^\rev C-1),
 \qquad M^\rev=zM.
\end{equation}
Interchanging $L$ and $R$ throughout a word interchanges its block sizes and leaves $M$ unchanged.
\end{lemma}
\begin{proof}
Let
\[
 B=\begin{pmatrix}z-1&z\\-1&z-1\end{pmatrix},\qquad
 \mathsf L=\begin{pmatrix}z&z\\0&1\end{pmatrix},\qquad
 \mathsf R=\begin{pmatrix}z&0\\1&1\end{pmatrix}.
\]
For $T^\dagger=(T(z^{-1}))^{\mathsf T}$, multiplication gives
$\mathsf L^\dagger B\mathsf L=B=\mathsf R^\dagger B\mathsf R$.
At the seed,
\[
 (A^\rev,C^\rev)B(A,C)^{\mathsf T}=z-1.
\]
Expanding this conserved expression proves the first identity in~\eqref{eq:conserved}; substituting into~\eqref{eq:M} gives the second.

The histograms for the letter-interchanged word are
\[
 (\widetilde A,\widetilde C)=z^n(C^\rev,A^\rev).
\]
This holds at the seed and is preserved by the two moves, with their letters interchanged. Formula~\eqref{eq:M} is unchanged, and evaluation at $z=1$ interchanges the blocks.
\end{proof}

The least exponent of $M$ is $-n$, with coefficient one, contributed by $A^\rev C$. By symmetry its greatest exponent is $n-1$, also with coefficient one. Thus
\begin{equation}\label{eq:F}
 F=z^nM=\sum_{j=0}^{2n-1}m_jz^j,\qquad
 m_j=m_{2n-1-j},\qquad m_0=1.
\end{equation}
The actual eigenvalue support is $[-N,N+1]$ once positivity of all intervening coefficients is established below.

\section{Subtraction-free profiles and strict unimodality}\label{sec:positive}
We track the rising-half first differences together with three auxiliary profiles. Their subtraction-free recursion will prove strict unimodality.

Use common-degree reciprocals
\[
 \bar A=z^nA^\rev,\qquad\bar C=z^nC^\rev,\qquad U=A+C.
\]
Define the rising-half difference polynomial
\begin{equation}\label{eq:D}
 D=\sum_{j=0}^{n-1}d_jz^j,\qquad d_j=m_j-m_{j-1},\qquad m_{-1}=0.
\end{equation}
At this point no positivity of $D$ is assumed. The spectral formula and symmetry give
\begin{equation}\label{eq:skew}
 \bar A C-A\bar C=(1-z)F=D-z^{2n}D(z^{-1}).
\end{equation}
The reflected part starts at degree $n+1$.

For an ordinary polynomial, $[P]_{\le n}$ means deletion of all terms of degree greater than $n$. Introduce
\begin{align}
 J&=\cut{(1-z^2)A\bar C}{n},\label{eq:J}\\
 \PA&=\cut{(1-z^2)A\bar A}{n},\qquad
 \PC=\cut{(1-z^2)C\bar C}{n},\label{eq:profiles}\\
 H&=\cut{(1-z^2)U(\bar A+\bar C)}{n}.\label{eq:Hdef}
\end{align}
These expressions contain subtractions. Their positivity will follow from their recursion, rather than from an assumed property of the individual histograms.

\begin{lemma}[Positive reconstruction]\label{lem:reconstruction}
For every generated state,
\begin{align}
 \PA+\PC+J&=\cut{z(1+z)D}{n}+z^n,\label{eq:linearcompat}\\
 H&=(1+z)D+J+z^n.\label{eq:Hpositive}
\end{align}
The right side of~\eqref{eq:Hpositive} already has degree at most $n$.
\end{lemma}
\begin{proof}
Write $p=A\bar A$, $c=C\bar C$, $q=A\bar C$, and $r=\bar A C$. Then
\[
 F=p+r+c-z^n,\qquad r-q=(1-z)F,
\]
so
\begin{equation}\label{eq:rawcompat}
 p+c+q=zF+z^n.
\end{equation}
Multiply by $1-z^2$ and truncate through degree $n$. The term involving $F$ becomes $[z(1+z)D]_{\le n}$ by~\eqref{eq:skew}, and the trace term becomes $z^n$. This proves~\eqref{eq:linearcompat}.

Expanding $U(\bar A+\bar C)=p+c+q+r$ and using~\eqref{eq:skew} gives
\[
 H=\cut{\PA+\PC+2J+(1-z^2)D}{n}.
\]
Substitution of~\eqref{eq:linearcompat} cancels the two $z^2D$ terms and gives~\eqref{eq:Hpositive}.
\end{proof}

The marker $z^n$ is forced by the trace-zero correction. Omitting the scalar $-1$ from~\eqref{eq:M} would change this boundary term.

\begin{proposition}[Closed positive dynamics]\label{prop:dynamics}
At $n=1$ the state is
\[
 (D,J,\PA,\PC)=(1,0,z,z).
\]
With $H$ given by~\eqref{eq:Hpositive}, its exact child states are
\begin{align}
 L:\ (D',J',\PA',\PC')
 &=\left(D+J+\PC,\ \cut{z^2(J+\PC)}{n+1},\ zH,\ z\PC\right),\label{eq:profileL}\\
 R:\ (D',J',\PA',\PC')
 &=\left(D+J+\PA,\ \cut{z^2(J+\PA)}{n+1},\ z\PA,\ zH\right).\label{eq:profileR}
\end{align}
The new parameter is $n+1$. The old $D$ is padded with a zero at degree $n$.
\end{proposition}
\begin{proof}
The quadruple $(A,C,\bar A,\bar C)$ changes to
\[
 (zU,C,\bar A+\bar C,z\bar C)
 \quad\text{or}\quad
 (zA,U,\bar A,z(\bar A+\bar C)).
\]
The changes in~\eqref{eq:skew} are respectively $(1-z^2)\bar C U$ and $(1-z^2)A(\bar A+\bar C)$. Truncation gives $D'=D+J+\PC$ or $D'=D+J+\PA$.

The products defining the other three profiles acquire, under $L$, the forms
\[
 z^2U\bar C,\quad zU(\bar A+\bar C),\quad zC\bar C,
\]
and under $R$ the forms $z^2A(\bar A+\bar C)$, $zA\bar A$, and $zU(\bar A+\bar C)$. Multiply by $1-z^2$ and truncate at $n+1$. Old terms above $n$ cannot enter that cutoff because every product acquires at least one factor of $z$. The $z^2$ term requires the explicit truncation shown. This proves both formulas.
\end{proof}

Every operation in~\eqref{eq:Hpositive}, \eqref{eq:profileL}, and~\eqref{eq:profileR} is an addition, a monomial shift, or a truncation. The following simultaneous induction makes the strictness information explicit.

\begin{theorem}[Positive differences and increment support]\label{thm:positive}
Write $\ord P$ for the least exponent with nonzero coefficient in $P$, and put $\alpha=\ord A$, $\gamma=\ord\bar C$. Every generated state satisfies:
\begin{enumerate}[label=\textup{(\roman*)},nosep]
\item $d_0=1$ and $d_j\ge1$ for $0\le j\le n-1$;
\item $\PA$ has support $[\alpha,n]$, first coefficient one, and all later coefficients at least two; $\PC$ has the analogous property on $[\gamma,n]$;
\item $J=0$ if $\alpha+\gamma>n$, and otherwise has the same suffix property on $[\alpha+\gamma,n]$;
\item $H_0=1$ and $H_j\ge2$ for $1\le j\le n$.
\end{enumerate}
A \emph{run} is a maximal consecutive string of identical letters in a Euclidean word. For an appended letter $x$, let $\rho$ be the final-run length of the child word. If $E_x=D_{wx}-D_w=\sum_{j=0}^n e_jz^j$, then
\begin{equation}\label{eq:suffix}
 e_j=0\ (j<\rho),\qquad e_\rho=1,\qquad e_j\ge2\ (\rho<j\le n).
\end{equation}
In particular $D_{wx}\succeq D_w$ coefficientwise.
\end{theorem}
\begin{proof}
At the seed, $\alpha=\gamma=1$, $D=1$, $J=0$, $\PA=\PC=z$, and $H=1+2z$. Assume the assertions at a state. Equation~\eqref{eq:Hpositive} gives (iv): two consecutive positive $D$ coefficients supply the internal bound, while $d_{n-1}$ and the marker $z^n$ supply the bound at degree $n$.

Under $L$, the increment $E_L=J+\PC$ starts at $\gamma$ with coefficient one, since $J$ starts strictly later or is zero. All its later coefficients through $n$ are at least two. Thus $D'=D+E_L$ retains positivity and acquires a positive coefficient at degree $n$. The order parameters satisfy
\[
 \alpha'=1,\qquad\gamma'=\gamma+1.
\]
The profiles $\PA'=zH$ and $\PC'=z\PC$ have the required suffixes. Finally $J'=[z^2E_L]_{\le n+1}$ starts at $\gamma+2=\alpha'+\gamma'$ when this index is retained, and is zero otherwise. Its coefficient bounds follow from those of $E_L$. This proves (i)--(iii) at the child, and~\eqref{eq:Hpositive} gives (iv). The right move is identical with $\PA,\alpha$ and $\PC,\gamma$ interchanged, using $\alpha'=\alpha+1$ and $\gamma'=1$.

For a left child, $\gamma$ is one when a new left run begins and is one more than the old final left-run length when that run continues. The parameter $\alpha$ plays the same role for a right child. Both equal one at the seed. Consequently the first nonzero increment index is exactly $\rho$, proving~\eqref{eq:suffix}.
\end{proof}

\begin{proof}[Proof of Theorem~\ref{thm:mainunimodal}]
Theorem~\ref{thm:positive} gives $m_j-m_{j-1}\ge1$ for $1\le j\le n-1$. Equation~\eqref{eq:F} supplies the reflected half and the two equal middle coefficients. It also gives $m_j\ge j+1$ on the rising half and therefore no gap in the support. The central reversed exponents are $-1,0$, corresponding to actual eigenvalues $1,0$.
\end{proof}

\section{Finite verification for a fixed residue class}\label{sec:rays}
Fix a coprime initial pair and repeatedly add the second block to the first.  Algebraically, the second histogram stays fixed while the first evolves by
\[
A_{r+1}=z(A_r+C).
\]
After a bounded number of such steps, the only new part of the spectrum is a translated edge profile. Hence no genuinely new log-concavity test can appear.

Fix a genuine seed $(A,C)$, after a common monomial shift so that both supports lie in $[0,\ell]$, with $\ell$ a nonnegative integer. Coefficients of Laurent polynomials are understood to be zero outside their supports. Put $a=A(1)$, $b=C(1)$ and
\[
 X=A^\rev A,\qquad Y=C^\rev C,\qquad V=A^\rev C.
\]
After $r$ left moves,
\begin{equation}\label{eq:rayhist}
 A_r=z^rA+zI_rC,\qquad C_r=C.
\end{equation}
Let $M^{[r]}=M(A_r,C)$. Thus $M^{[r]}$ encodes the complete Ooms spectrum of $W(a+rb,b)$ in reversed degrees.

\begin{proposition}[Separated discrete curvature]\label{prop:curvature}
For every $r\ge0$,
\begin{equation}\label{eq:rayfull}
 M^{[r]}=X-1+Y(1+I_rI_r^\rev+z^{-r}I_r)
            +VI_{r+1}^\rev+V^\rev I_r.
\end{equation}
Define
\[
 \mathscr L=2-z-z^{-1},\quad \beta=1+z^{-1},\quad
 B=-Y+(z^{-1}-1)V^\rev.
\]
The operator $\mathscr L$ is simply the discrete second-difference operator: the coefficient of $z^k$ in $\mathscr L M$ is $2M_k-M_{k-1}-M_{k+1}$.
Then
\begin{equation}\label{eq:curvature}
 \mathscr L M^{[r]}=\beta Y+z^rB+z^{-r-1}B^\rev,
\end{equation}
where $\supp Y\subseteq[-\ell,\ell]$ and $\supp B\subseteq[-\ell-1,\ell]$.
\end{proposition}
\begin{proof}
Expansion of~\eqref{eq:M} using~\eqref{eq:rayhist} gives~\eqref{eq:rayfull}; the cross-term coefficients use
\[
 z^{1-r}I_r+z^{-r}=I_{r+1}^\rev,\qquad
 z^{r-1}I_r^\rev=I_r,\qquad z^{-1}I_r^\rev=z^{-r}I_r.
\]
The geometric-sum identities
\begin{align*}
 \mathscr L(1+I_rI_r^\rev+z^{-r}I_r)&=3-z-z^r-z^{-r-1},\\
 \mathscr L I_{r+1}^\rev&=(1-z)+z^{-r}(1-z^{-1}),\\
 \mathscr L I_r&=(1-z^{-1})+z^{r-1}(1-z)
\end{align*}
group the moving terms as in~\eqref{eq:curvature}. The stationary terms reduce to $\beta Y$ because~\eqref{eq:conserved} implies
\[
 \mathscr L(X+Y-1)+(1-z)V+(1-z^{-1})V^\rev=0.
\]
All identities remain valid at $r=0$. The support bounds follow from the seed supports.
\end{proof}

For $Y_j=[z^j]Y$ and $V_j=[z^j]V$, define the fixed edge sequence
\begin{equation}\label{eq:edge}
 \mathcal E_t=\sum_{j>t}(j-t)Y_j+\sum_{j\le-t-1}V_j,
 \qquad t\in\Z.
\end{equation}

\begin{lemma}[A translating edge with an affine continuation]\label{lem:edge}
For every $r\ge0$ and every integer $k>\ell$,
\begin{equation}\label{eq:edgematch}
 [z^k]M^{[r]}=\mathcal E_{k-r}.
\end{equation}
Furthermore,
\begin{align}
 \mathcal E_t-\mathcal E_{t+1}&=\sum_{j\ge t+1}Y_j+V_{-t-1}\ge0,\label{eq:edgediff}\\
 \mathcal E_t&=0\quad(t\ge\ell),\label{eq:edgezero}\\
 \mathcal E_t&=ab-tb^2\quad(t\le-\ell-1).\label{eq:affine}
\end{align}
Hence
\begin{equation}\label{eq:b4}
 \mathcal E_t^2-\mathcal E_{t-1}\mathcal E_{t+1}=b^4
       \qquad(t\le-\ell-2).
\end{equation}
\end{lemma}
\begin{proof}
At $k>\ell$, the terms $X-1$, $Y$, and $VI_{r+1}^\rev$ in~\eqref{eq:rayfull} contribute zero. In the $Y$ convolution, $k-j>0$ on the support of $Y$, so the remaining coefficient is
\[
 \sum_jY_j\max\{0,r-k+j\}.
\]
The term $V^\rev I_r$ contributes $\sum_{-k\le j\le r-k-1}V_j$. Its lower restriction is redundant since $-k<-\ell$. These are exactly the two sums in~\eqref{eq:edge} at $t=k-r$.

Subtraction proves~\eqref{eq:edgediff}; the support bounds give~\eqref{eq:edgezero}. Since $Y^\rev=Y$,
\[
 \sum_jY_j=b^2,\qquad\sum_j jY_j=0,\qquad\sum_jV_j=ab.
\]
All terms contribute when $t\le-\ell-1$, proving~\eqref{eq:affine}. Three consecutive affine entries have defect $b^4$.
\end{proof}

\begin{theorem}[Stability after finitely many terms]\label{thm:stable}
Let $R_\ell=2\ell+4$. The following are equivalent:
\begin{enumerate}[label=\textup{(\roman*)},nosep]
\item $M^{[R_\ell]}$ is log-concave;
\item $M^{[r]}$ is log-concave for every $r\ge R_\ell$;
\item $\mathcal E_t^2\ge\mathcal E_{t-1}\mathcal E_{t+1}$ for $-\ell-1\le t\le\ell$.
\end{enumerate}
If $M^{[R_\ell]}$ is strictly internally log-concave, then so is every $M^{[r]}$ with $r\ge R_\ell$.
\end{theorem}
\begin{proof}
By symmetry, it is enough to check centers $k\ge0$. For $r\ge R_\ell$ and $0\le k\le\ell+1$, the two moving terms in~\eqref{eq:curvature} are absent: the first starts at least at $r-\ell-1\ge\ell+3$, and the second ends at most at $\ell-r\le-\ell-4$. Thus
\begin{equation}\label{eq:centralconcave}
 2M^{[r]}_k-M^{[r]}_{k-1}-M^{[r]}_{k+1}=Y_k+Y_{k+1}\ge0.
\end{equation}
Ordinary concavity of a nonnegative triple implies log-concavity.

For $k\ge\ell+2$, all three coefficient indices exceed $\ell$. Their inequality is therefore exactly the edge inequality at $t=k-r$. At $r=R_\ell$, this range starts at $t=-\ell-2$ and includes every potentially nontrivial edge inequality. Entries further to the left have defect $b^4$ by~\eqref{eq:b4}; entries at $t\ge\ell$ have zero right neighbor. Increasing $r$ exposes only additional affine inequalities. This proves (iii)$\Rightarrow$(ii), while (ii)$\Rightarrow$(i) is immediate. For (i)$\Rightarrow$(iii), use $k=R_\ell+t\ge\ell+3$ in~\eqref{eq:edgematch} for all three entries.

For strictness in the central window, put $\delta_k=M^{[r]}_k-M^{[r]}_{k+1}$. Symmetry gives $\delta_{-1}=0$, and~\eqref{eq:centralconcave} gives
\[
 \delta_k-\delta_{k-1}=Y_k+Y_{k+1},\qquad
 \delta_0=Y_0+Y_1>0.
\]
Consequently all $\delta_k$ in that window are positive, and
\begin{equation}\label{eq:strictcenter}
 (M^{[r]}_k)^2-M^{[r]}_{k-1}M^{[r]}_{k+1}
 =M^{[r]}_k(Y_k+Y_{k+1})+\delta_{k-1}\delta_k>0
\end{equation}
at each internal index. For $k=0$, the first term is positive; for $k\ge1$, the product is positive. The remaining internal edge inequalities are copies of internal inequalities at the cutoff or have the positive affine defect $b^4$. This proves the strict assertion.
\end{proof}

\begin{proof}[Proof of Theorem~\ref{thm:mainray}]
Combine Theorem~\ref{thm:stable} with the finitely many values $r<R_\ell$. The analogous statement obtained by repeatedly adding the first block to the second follows by interchanging blocks and letters using Lemma~\ref{lem:symmetry}.
\end{proof}

\begin{remark}
No log-concavity of $A$, $C$, their autocorrelations, or their mixed coefficients is assumed. The edge inequalities in Theorem~\ref{thm:stable}(iii) are genuine seed-dependent conditions. A failure of one of them persists at every $r\ge R_\ell$.
\end{remark}

\section{Symbolic reduction and exact finite verification through smaller block 128}\label{sec:certificate}
We apply Theorem~\ref{thm:mainray} to every residue class with smaller block at most $128$. The required finite checks use exact integer arithmetic.

\begin{proof}[Proof of Corollary~\ref{thm:main128}]
Order the blocks so that $a\ge b$. For $b>1$, write $a=s+rb$ with $1\le s<b$, $\gcd(s,b)=1$, and $r\ge1$. For $b=1$, use the seed $(1,1)$ and $r=a-1$. It is therefore enough to check the families starting at
\begin{equation}\label{eq:seedset}
 \{(1,1)\}\ \cup\ 
 \{(s,b):2\le b\le128,\ 1\le s<b,\ \gcd(s,b)=1\}.
\end{equation}
There are $5{,}022$ such seeds.  In particular, the modulus $9$ that was missing from Corollary~\ref{cor:smallblocks} is fully included here.  Its Frobenius residue classes are $\pm1,\pm2,\pm4$; Theorem~\ref{thm:residues} handles $\pm1,\pm2$, while the present finite reduction also handles the two classes $\pm4$.  Thus $b=9$ is omitted only from the purely symbolic small-residue corollary, not from Corollary~\ref{thm:main128}.

For a seed $(s,b)$ with $s<b$, let its Euclidean word length be $N$. Each move increases the larger block by at least one, so $N\le b-1$. Its canonical histograms have degree at most $\ell=N+1\le b$. The seed $(1,1)$ has $\ell=1$. Thus
\[
 R_\ell=2\ell+4\le2b+4\le260.
\]
The exact computations described below check every required spectrum at $0\le r\le R_\ell$ for every seed in~\eqref{eq:seedset}. Every internal inequality is strict. Theorem~\ref{thm:mainray} supplies all larger $r$, and block interchange supplies both block orders.
\end{proof}

\subsection{The finite check determined by the theorem}
For each $b\le128$ and each $s$ with $1\le s<b$ and $\gcd(s,b)=1$, we construct the initial pair $W(s,b)$ and its two exact histogram polynomials.  If their common normalized support lies in $[0,\ell]$, Theorem~\ref{thm:mainray} proves that only
\[
q=0,1,\ldots,2\ell+4
\]
need be considered in the family $W(s+qb,b)$.  For each of these finitely many values, the complete multiplicity polynomial is obtained from~\eqref{eq:M}, and every internal inequality
\[
m_k^2>m_{k-1}m_{k+1}
\]
is tested by exact integer comparison.  Passing these finite inequalities proves the whole infinite family by Theorem~\ref{thm:mainray}.

An independent reconstruction from the directed meander gives the same spectra. The verification archive contains the source code, complete output records, and reproduction instructions.

\begin{example}
The seeds $(4,9)$ and $(5,9)$ both have $\ell=6$ and cutoff $16$. For each initial pair, the computation checks $17$ spectra and $442$ internal inequalities. The respective least defects are $2$ and $7$. Theorem~\ref{thm:mainray} then proves strict internal log-concavity of $W(9r+4,9)$ and $W(9r+5,9)$ for every $r\ge0$, not merely for $0\le r\le16$.
\end{example}

\section{The remaining log-concavity problem}\label{sec:questions}
The conjecture that every Frobenius type-$A$ maximal parabolic has a log-concave Ooms spectrum was stated by Mayers and Russoniello~\cite[Conjecture~75 in the arXiv version]{MR}. The unrestricted conjecture is not proved here.

The positive first-difference recursion suggests a stronger sufficient condition:
\begin{conjecture}\label{conj:D}
For every Euclidean word, the coefficients of $D$ in~\eqref{eq:D} are log-concave.
\end{conjecture}
Its positivity has already been proved in Theorem~\ref{thm:positive}. To see the implication for the spectrum, let $d_0,\ldots,d_t$ be positive and log-concave, and let $m_j=\sum_{i=0}^j d_i$. Decreasing adjacent ratios give $d_{j+1}d_i\le d_jd_{i+1}$ for $i<j$, whence
\[
 m_j^2-m_{j-1}m_{j+1}
 =d_jm_j-d_{j+1}m_{j-1}\ge d_0d_j>0.
\]
The middle defect is $m_{n-1}d_{n-1}>0$, and symmetry gives the other half. Thus Conjecture~\ref{conj:D} would prove strict internal log-concavity for all coprime blocks.

Subtraction-free generation alone does not establish this conjecture. A sum of positive log-concave sequences need not be log-concave, and the auxiliary profiles are coupled through their common histograms. No general log-concavity-preservation claim is made for arbitrary nonnegative tuples in~\eqref{eq:profileL}--\eqref{eq:profileR}.

\appendix
\section{Verification of the paired residue chains}\label{app:lift}
We supply the edge and counting details for Proposition~\ref{prop:lift}. Suppose $i<t=\tau(i)$ and $u_i=u$, so $u_t=u+1$. Write $q=q_i=q_t$. Formula~\eqref{eq:minlift} becomes
\begin{equation}\label{eq:paired}
 P_i(j)=u+\min(2j,2q-2j+2),\qquad
 P_t(j)=u+1+\min(2j,2q-2j),\qquad0\le j\le q.
\end{equation}
The top mate of $i+jb$ is $t+(q-j)b$. The first vertex is smaller exactly when $2j\le q$. In that case the values are $u+2j$ and $u+2j+1$. When $2j>q$, the values are $u+2q-2j+2$ and $u+2q-2j+1$, in reversed vertex order. Both cases satisfy the top relation in~\eqref{eq:potential}.

For $1\le j\le q$, the bottom mate is $t+(q+1-j)b$. The first vertex is smaller exactly when $2j\le q+1$. The corresponding values are $u+2j,u+2j-1$ in that case, and $u+2q-2j+2,u+2q-2j+3$ otherwise. These are the required bottom relations. Listing an edge from the other residue chain only reverses the same pair.

At a fixed residue $i=\tau(i)$, the values are
\[
 P_i(j)=u_i+\min(2j,2q_i-2j+1).
\]
The top pair has indices $j,q_i-j$, and the bottom pair has indices $j,q_i+1-j$. When the first index is smaller, the potential differences are respectively $+1$ and $-1$. Equal indices correspond to a fixed reflection vertex and no edge. This verifies all internal edges, including fixed residues.

A bottom edge from residue position $i$ to the second block joins the values $u_i,u_i-1$ by~\eqref{eq:smalllift}. A top edge in the second block is a reversed bottom edge of the seed. Finally, the outer bottom seed edge gives $u_1=u_b+1=1$, so the last parent value is zero. The proposed values therefore satisfy every edge relation and the normalization; uniqueness on the connected path proves the potential formula.

For completeness, the histogram of the two sequences in~\eqref{eq:paired} is
\[
 z^u(1+z)I_{q+1}.
\]
If $q=2h$, the first sequence contributes zero once and every positive even integer through $2h$ twice, after subtracting $u$; the second contributes the odd integers through $2h-1$ twice and $2h+1$ once. If $q=2h+1$, the first contributes zero once, positive even integers through $2h$ twice, and $2h+2$ once; the second contributes every odd integer through $2h+1$ twice. In either case the endpoints occur once and all intervening values twice. At a fixed residue, the two branches together give each integer $u_i,\ldots,u_i+q_i$ once, with histogram $z^{u_i}I_{q_i+1}$. These are the orbit counts used in~\eqref{eq:lift}.

\section{Exact finite verification}\label{app:protocol}
The calculation for Corollary~\ref{thm:main128} proceeds as follows.
\begin{enumerate}[label=\textup{\arabic*.},leftmargin=2em]
\item Enumerate the initial pairs in~\eqref{eq:seedset} and construct their exact histogram polynomials $A,C$.
\item Normalize their common support to $[0,\ell]$.  Theorem~\ref{thm:mainray} reduces the entire family $W(s+rb,b)$ to the finitely many values $0\le r\le2\ell+4$.
\item For each of those values, compute the complete multiplicity polynomial from~\eqref{eq:M} and test every internal inequality
\[
 m_k^2>m_{k-1}m_{k+1}
\]
by exact integer arithmetic.
\item Apply Theorem~\ref{thm:mainray}; all larger values of $r$ then follow symbolically.
\end{enumerate}

For $b\le128$ this amounts to $5{,}022$ initial pairs and $216{,}868$ finite spectra.  Every required inequality is strict.  A second, independent construction recovers the same spectra directly from the meander and the separated-curvature identity.

\section*{Acknowledgement}
The author thanks Anthony Giaquinto for bringing~\cite{GILM} to his attention while the present manuscript was being prepared for posting.

\section*{Data and code availability}
The mathematical supplement~\cite{SR} is supplied as the second top-level source document and is appended to the arXiv PDF. It forms part of the proofs of Theorem~\ref{thm:residues}. The accompanying ancillary verification archive contains the exact-arithmetic programs, finite-verification records, and reproduction instructions.
\ifdefempty{\CodeRepository}{}{The same material is available in the \href{\CodeRepository}{code repository}.}

\end{document}

\maketitle
\noindent
This supplement supplies the full proofs of Theorem~1.2 of the accompanying manuscript~\cite{Main}.
Parts A, B, and C treat residues $\pm1$, $\pm2$, and $\pm3$, respectively.
The interval inequalities, rational identities, parameter ranges, and finite base calculations are given explicitly.
The programs in the accompanying verification archive provide separate checks of these calculations. This supplement forms part of the proofs in the main paper.
The proof of the smaller-block-$128$ result is in the main manuscript and is not used here.

\medskip
\noindent\textbf{Guide to the proof.}
Part A proves two interval-polynomial inequalities and applies exact spectral factorizations.
Parts B and C give coefficient formulas, prove that the second differences decrease after an explicit initial segment, and check that segment in finitely many parameter ranges.
The coefficients of the relevant quartics and all finite base checks are printed below.

\tableofcontents
\clearpage
\part*{Part A: Residues plus or minus 1}
\addcontentsline{toc}{part}{Part A: Residues plus or minus 1}
\setcounter{section}{0}
\renewcommand{\thesection}{A.\arabic{section}}
\renewcommand{\theHsection}{\thesection}
\section{Conventions}

Throughout, $W(a,b)$ denotes the maximal-parabolic type-$A$ seaweed with
top blocks $(a,b)$ and one bottom block $(a+b)$.  All block sizes are positive
and all spectra under consideration are Frobenius spectra:
\[
 \gcd(a,b)=1.
\]
Put $F^\vee(z)=F(z^{-1})$ for a Laurent polynomial, and
\[
 I_n(z)=1+z+\cdots+z^{n-1}\quad(n\ge1),\qquad I_0=0.
\]
Orient top meander arcs from right to left and bottom arcs from left to
right.  The potential $\psi$ satisfies $\psi(u)=\psi(v)+1$ along a directed
edge $u\to v$.  Thus, for endpoints $x<y$, a top arc has
$\psi(y)=\psi(x)+1$ and a bottom arc has $\psi(x)=\psi(y)+1$.
Normalize at the last vertex, or apply a common translation to all heights.

If $A,C$ are the two block histograms, the reversed multiplicity polynomial is
\begin{equation}\label{A:eq:block}
 M=\M(A,C)=A^\vee A+A^\vee C+C^\vee C-1.
\end{equation}
The coefficient at degree $k$ is the multiplicity of the actual Ooms
eigenvalue $-k$.  This is the directed-meander/principal-element
identification; see \cite{CHM,MR}.  The scalar correction in
\eqref{A:eq:block} removes one zero eigenvalue when imposing trace zero.
The known symmetry reads
\begin{equation}\label{A:eq:sym}
 M^\vee=zM,\qquad M_k=M_{-1-k}.
\end{equation}
A common monomial factor in $A,C$ has no effect on $M$; a monomial factor in
$M$ only translates its coefficient word.  Reversal likewise preserves
log-concavity.

For a nonnegative finite coefficient word, extended by zero, put
\[
 D_j(f)=f_j^2-f_{j-1}f_{j+1}.
\]
Log-concavity means $D_j(f)\ge0$ at every degree.  All polynomials to which
we apply convolution preservation below have contiguous positive support.

\section{Euclidean histograms and the two endpoint residues}

\begin{lemma}[All-parameter histogram lifting]\label{A:lem:histlift}
Fix $1\le r<b$ with $\gcd(r,b)=1$.  Let $A,C$ be the two block
histograms of the seed $W(r,b-r)$ and $U=A+C$.  In a common compatible
normalization, the two histograms of $W(qb+r,b)$ are
\begin{equation}\label{A:eq:histlift}
 H_q=A I_{q+1}+C I_q,\qquad G=z^{-1}U,\qquad q\ge0.
\end{equation}
The zero-enlargement case $q=0$ is the actual meander of $W(r,b)$.
\end{lemma}
\begin{proof}
We give the residue-chain verification used in \cite[Proposition~2.2]{Main}
to make the parameter ranges explicit.  Let $u_s$ be the seed potential,
$1\le s\le b$, and let
\[
 \tau(s)=\begin{cases}r+1-s,&s\le r,\\ b+r+1-s,&s>r.\end{cases}
\]
For $q\ge1$, the chain $s,s+b,\ldots,s+Q_sb$ in the large block has
$Q_s=q$ for $s\le r$ and $Q_s=q-1$ otherwise.  The top involution preserves
$Q_s$.  Assign
\begin{align}
 \psi(s+jb)&=\min\{u_s+2j,u_{\tau(s)}+1+2(Q_s-j)\},
       &&0\le j\le Q_s,\label{A:eq:lift-potential}\\
 \psi(qb+r+t)&=u_{b+1-t}-1,&&1\le t\le b.\label{A:eq:lift-small}
\end{align}

For a nontrivial seed top pair $s<t=\tau(s)$, put $u_s=u$, so
$u_t=u+1$, and write $Q=Q_s=Q_t$.  Its two residue-chain sequences are
\[
 x_j=u+\min(2j,2Q-2j+2),\qquad
 y_j=u+1+\min(2j,2Q-2j).
\]
The top mate of $s+jb$ is $t+(Q-j)b$.  The former is smaller exactly when
$2j\le Q$; in that case their values are $u+2j,u+2j+1$.
In the other order the values are $u+2Q-2j+2,u+2Q-2j+1$.
The bottom mate, for $j\ge1$, is $t+(Q+1-j)b$.
Here the former vertex is smaller exactly when $2j\le Q+1$; the values
are then $u+2j,u+2j-1$, and otherwise are
$u+2Q-2j+2,u+2Q-2j+3$.  These are the required top and bottom equations.
Listing an edge from the other residue reverses the same verified pair.

For a fixed residue $s=\tau(s)$ the sequence is
$u_s+\min(2j,2Q_s-2j+1)$.  Its top pairs have indices $j,Q_s-j$ and its
bottom pairs $j,Q_s+1-j$.  For a smaller first index the corresponding
values differ by $+1$ for a top edge and $-1$ for a bottom edge.
Equal indices are fixed vertices and have no edge.  Thus these cases are
also covered.

The boundary bottom edge at residue $s$ joins values $u_s,u_s-1$.
Reversing a small-block top edge in \eqref{A:eq:lift-small} gives a seed bottom
edge.  This verifies every remaining edge.  With $u_b=0$, the outer seed
bottom arc gives $u_1=1$, so the last parent value is zero.  Uniqueness on
the connected meander identifies the assigned potential.

For a paired orbit the joint histogram is $z^u(1+z)I_{Q+1}$: its two
endpoint heights occur once and every intermediate height twice.  Indeed,
if $Q=2d$, the $x$-chain gives zero once and every positive even height
through $2d$ twice, while the $y$-chain gives odd heights through $2d-1$
twice and $2d+1$ once, after translation by $u$.  If $Q=2d+1$, the
$x$-chain gives zero once, the intervening positive even heights twice,
and $2d+2$ once; the $y$-chain gives each odd height through $2d+1$ twice.
For a fixed residue, the even and odd branches together give each height
$u_s,\ldots,u_s+Q_s$ once.  Therefore each first seed block is multiplied
by $I_{q+1}$ and each second seed block by $I_q$, proving \eqref{A:eq:histlift}
for $q\ge1$.

For $q=0$, assign $\psi(i)=u_i$ for $1\le i\le r$ and
$\psi(r+t)=u_{b+1-t}-1$ for $1\le t\le b$.
First-block top edges are seed top edges, and second-block top edges are
reversed seed bottom edges.  Bottom edges meeting the first block join
$u_i,u_i-1$.  For a bottom edge entirely in the second block, its endpoints
$r+t<r+t'$ satisfy $t+t'=b-r+1$; the corresponding seed indices
$b+1-t,b+1-t'$ are in the seed second top block, sum to $b+r+1$, and are
in reverse order.  The seed top equation is exactly the parent bottom
equation.  The normalization again holds.  Thus the actual $q=0$ histograms
are $A,z^{-1}U$, as stated.
\end{proof}

\begin{proposition}[Endpoint-residue histogram formulas]\label{A:prop:endhist}
For every $b\ge2$ and $m\ge1$, compatible block histograms are
\begin{align}
 W(mb+1,b):&\quad H=I_bI_m+z^{m+b-1},\qquad G=z^{-1}I_b;
       \label{A:eq:histplus}\\
 W(mb-1,b):&\quad H=I_bI_m-z^{m-1},\qquad G=z^{-1}I_b.
       \label{A:eq:histminus}
\end{align}
The subtraction in \eqref{A:eq:histminus} leaves a nonnegative histogram.
\end{proposition}
\begin{proof}
For $W(c,1)$, the first block has potential heights $1,\ldots,c$ once
each and the last height is zero.  The successive vertices
$1,c,2,c-1,\ldots$ have heights $1,2,3,4,\ldots$, directly from the
alternating top and bottom equations.  Its histograms are $zI_c,1$.
Reversing the vertices and replacing the potential by $1-\psi$ gives
histograms $zC^\vee,zA^\vee$ for the reversed block order.
Consequently a common normalization for $W(1,b-1)$ is
\[
 A=z^{b-1},\qquad C=I_{b-1},\qquad A+C=I_b.
\]
Lemma~\ref{A:lem:histlift} gives
\[
 H=z^{b-1}I_{m+1}+I_{b-1}I_m=I_bI_m+z^{m+b-1}.
\]
For the other family use the seed $W(b-1,1)$ with
$A=zI_{b-1},C=1$ and enlargement index $q=m-1\ge0$.  Then
\[
 H=zI_{b-1}I_m+I_{m-1}=I_bI_m-z^{m-1}.
\]
The preceding expression as a sum of two nonnegative polynomials also
proves the last assertion.
\end{proof}

\section{Two coefficient inequalities}

\begin{lemma}[Convolution preservation]\label{A:lem:convolution}
The product of two polynomials with nonnegative, log-concave coefficient
words and no internal zeros has the same properties.
\end{lemma}
\begin{proof}
Extend a coefficient word $f$ by zero to all of $\Z$.  Log-concavity and
contiguous support imply that the ratios of consecutive positive entries
are nonincreasing.  Hence, for positive integers $p,q$ and every integer
$t$,
\[
 f_{t+p}f_{t+q}\ge f_t f_{t+p+q}.
\]
If the right side is positive, this follows by comparing products of $p$
adjacent ratios; if it is zero the inequality is immediate.
Thus every $2\times2$ minor of the Toeplitz matrix
$T(f)_{ij}=f_{j-i}$ is nonnegative.  For two coefficient words,
$T(f*g)=T(f)T(g)$.  The Cauchy--Binet formula writes every $2\times2$
minor of this product as a sum of products of such minors.  Only finitely
many summands contribute because the words have finite support.  In
particular the adjacent minors of $T(f*g)$ are nonnegative, which is
log-concavity.  The support of the convolution is the sum of the two
support intervals and all its coefficients there are positive.
\end{proof}

\begin{lemma}[Three intervals with an added interval]\label{A:lem:three}
Let $p,q,r\ge1$ and $L=p+q+r$.  Then
\begin{equation}\label{A:eq:threebox}
 S(z)=I_L+zI_pI_qI_r
\end{equation}
has a symmetric, positive, log-concave coefficient word.  If
$p\le q\le r$, it is ordinarily concave at each index
\begin{equation}\label{A:eq:latehalf}
 p\le j\le\left\lfloor\frac{L-1}{2}\right\rfloor.
\end{equation}
Here ordinary concavity is asserted only at these specified indices, not
on the whole zero-extended word.
\end{lemma}
\begin{proof}
Both summands are symmetric about $(L-1)/2$, and $I_L$ makes every
coefficient in degrees $0,\ldots,L-1$ positive.  Permute the three lengths
so that $p\le q\le r$, and put $T=I_pI_qI_r$.
It suffices by reflection to check $1\le j\le\lfloor(L-1)/2\rfloor$.

For $1\le j<p$, the coefficients at the three relevant indices are
\[
 S_i=f(i):=1+\frac{i(i+1)}2,\qquad i=j-1,j,j+1.
\]
Indeed, the coefficient of $z^{i-1}$ in $T$ still counts unrestricted
triples of nonnegative integers.  Direct calculation gives
\begin{equation}\label{A:eq:three-early-defect}
 f(j)^2-f(j-1)f(j+1)=\frac{(j-1)(j+2)}2\ge0.
\end{equation}

For the range \eqref{A:eq:latehalf},
\begin{align}
 S_{j+1}-2S_j+S_{j-1}
 &=[z^j](1-z)^2T\nonumber\\
 &=[z^j]\frac{(1-z^p)(1-z^q)(1-z^r)}{1-z}\nonumber\\
 &=1-\ind_{j\ge p}-\ind_{j\ge q}-\ind_{j\ge r}
                  +\ind_{j\ge p+q}.\label{A:eq:three-second}
\end{align}
The other pair terms are absent because $j<p+r$ and $j<q+r$ on this
half of the support.  If $j<p+q$, the indicator $\ind_{j\ge p}=1$
already makes \eqref{A:eq:three-second} nonpositive.  If $j\ge p+q$,
the first two singleton terms cancel the constant and the pair term,
leaving $-\ind_{j\ge r}\le0$.  Thus ordinary concavity holds in this
range.  For nonnegative entries it implies log-concavity by
\[
 S_j\ge\tfrac12(S_{j-1}+S_{j+1})\ge\sqrt{S_{j-1}S_{j+1}}.
\]
Together with \eqref{A:eq:three-early-defect} and reflection, this proves
every inequality, including the central one or central pair.  The outside
endpoint inequalities are automatic because the adjacent outside
coefficients vanish.
\end{proof}

\begin{lemma}[The second endpoint-residue polynomial]\label{A:lem:minusbracket}
For $b\ge2$ and $m\ge1$, the polynomial
\begin{equation}\label{A:eq:minusbracket}
 F=I_{b+2m-1}+zI_bI_{m-1}I_m+zI_{b-2}I_{2m}
\end{equation}
has a symmetric, positive, log-concave coefficient word.
\end{lemma}
\begin{proof}
For $m=1$, geometric sums give
\[
 F=I_{b+1}+zI_{b-2}I_2=I_{b-1}I_3,
\]
so Lemma~\ref{A:lem:convolution} applies.  This includes $b=2$, with $I_0=0$.

Assume $m\ge2$, put $L=b+2m-1$ and $p=\min(b,m-1)$, and write
\[
 S=I_L+zI_bI_{m-1}I_m,\qquad R=I_{b-2}I_{2m},\qquad F=S+zR.
\]
The summands have common symmetry center $(L-1)/2$ and $F$ is positive
on $0,\ldots,L-1$.  If $b>2$, the coefficient word of $zR$ is a
trapezoid supported on $1,\ldots,L-2$ and is ordinarily concave at every
index from $1$ through $L-2$; at its support endpoints the adjacent
outside coefficients are zero and the same inequality holds.  If $b=2$
it is identically zero.  By Lemma~\ref{A:lem:three}, $S$, and consequently
$F$, is ordinarily concave for
$p\le j\le\lfloor(L-1)/2\rfloor$.

For $1\le j<p$, the three neighboring coefficients have the explicit form
\begin{equation}\label{A:eq:minus-early}
 F_i=f(i):=1+\frac{i(i+1)}2+\min(i,b-2),
 \qquad i=j-1,j,j+1.
\end{equation}
To verify this, note $i\le p\le m-1$ and $i\le b$: the three-box
coefficient is still its initial quadratic, while the two-box coefficient
is $\min(i,b-2)$.  Since $j<p\le b$, there are only three possibilities:
\[
\begin{array}{c|c}
 \text{index}&f(j)^2-f(j-1)f(j+1)\\ \hline
 1\le j\le b-3& (j+1)(j+2)/2\\
 j=b-2\ (j\ge1)&(b-1)^2\\
 j=b-1&(b-1)(b-2)/2.
\end{array}
\]
For the first row, $f(i)=\binom{i+2}{2}$ on the triple.
For the second, only the last member of this quadratic triple is
reduced by one.  For the third, the triple is
$f(i)=b-1+i(i+1)/2$.  These descriptions prove the displayed values
directly.  Empty index ranges are ignored; in particular the last row
also covers $b=2,j=1$ and gives zero.  All values are nonnegative.
Reflection and the ordinarily concave part already checked cover every
remaining interior inequality.
\end{proof}

\section{Exact spectral factorizations and log-concavity}

\begin{proposition}[Two spectral factorizations]\label{A:prop:factor}
Let $M^+_{b,m}$ and $M^-_{b,m}$ be the reversed multiplicity polynomials
of $W(mb+1,b)$ and $W(mb-1,b)$, respectively.  For all $b\ge2,m\ge1$,
\begin{align}
 z^{b+m}M^+_{b,m}
 &=I_b\bigl(I_{b+2m+1}+zI_bI_mI_{m+1}\bigr),\label{A:eq:plusfactor}\\
 z^{b+m-1}M^-_{b,m}
 &=I_b\bigl(I_{b+2m-1}+zI_bI_{m-1}I_m
                       +zI_{b-2}I_{2m}\bigr).\label{A:eq:minusfactor}
\end{align}
\end{proposition}
\begin{proof}
Put $B=I_b$ and $J=I_m$.  First take
$H=BJ+e$, $e=z^{b+m-1}$, and $G=z^{-1}B$, as in
\eqref{A:eq:histplus}.  The monomial autocorrelation $e^\vee e=1$
cancels the trace correction in \eqref{A:eq:block}.  Using
$I_n^\vee=z^{1-n}I_n$, expansion gives
\begin{align*}
 z^{b+m}\M(H,G)
 =B\bigl[z^2BJ^2+zBJ+z^{m+1}B+z^{m+b+1}J+zJ+1\bigr].
\end{align*}
The first two bracket terms equal $zBJ I_{m+1}$ because $1+zJ=I_{m+1}$.
The rest equal
\[
 z^{m+1}(B+z^bJ)+I_{m+1}
 =z^{m+1}I_{b+m}+I_{m+1}=I_{b+2m+1}.
\]
This proves \eqref{A:eq:plusfactor}.

For \eqref{A:eq:histminus}, take $H=BJ-e$ with $e=z^{m-1}$.
Again $e^\vee e$ cancels the trace correction, and expansion now gives
\begin{align*}
 z^{b+m-1}\M(H,G)
 =B\bigl[zBJ^2+BJ+z^mB-(z^m+z^b)J-z^{b-1}\bigr].
\end{align*}
The bracket can be rearranged as
\begin{equation}\label{A:eq:minusintermediate}
 zI_{b-1}I_mI_{m+1}+I_m^2-I_m+I_{b-1}.
\end{equation}
Indeed, subtracting \eqref{A:eq:minusintermediate} from the bracket and
using $(z-1)J=z^m-1$ reduces the difference to
$(1-z^b)J+(z^m-1)B=0$.
The remaining geometric-sum identity is
\begin{align*}
 zI_{b-1}I_mI_{m+1}+I_m^2-I_m+I_{b-1}
 =I_{b+2m-1}+zI_bI_{m-1}I_m+zI_{b-2}I_{2m}.
\end{align*}
It follows by replacing $I_n$ by $(1-z^n)/(1-z)$: multiplication by
$(1-z)^3$ gives on both sides
\begin{align*}
 &(1-z^{b+2m-1})(1-z)^2
 +z(1-z^b)(1-z^{m-1})(1-z^m)\\
 &\hspace{4em}+z(1-z^{b-2})(1-z^{2m})(1-z).
\end{align*}
All identities include $m=1$ and $b=2$, with zero-length interval
polynomials interpreted as zero.  This proves \eqref{A:eq:minusfactor}.
\end{proof}

\Needspace{12\baselineskip}
\begin{theorem}[Two unbounded log-concave families]\label{A:thm:main}
For every $b\ge2$ and $m\ge1$, the Ooms multiplicity sequences of
\[
 \boxed{W(mb+1,b)\quad\text{and}\quad W(mb-1,b)}
\]
are log-concave and therefore unimodal.  Every algebra in these families
is Frobenius.  There is no upper bound on $b$ or $m$.
\end{theorem}
\begin{proof}
Coprimality is immediate from $mb\pm1\equiv\pm1\pmod b$.
In \eqref{A:eq:plusfactor}, the bracket is
Lemma~\ref{A:lem:three} with lengths $b,m,m+1$.
The bracket in \eqref{A:eq:minusfactor} is
Lemma~\ref{A:lem:minusbracket}.  Each bracket and the factor $I_b$ have
positive contiguous support and are log-concave.
Lemma~\ref{A:lem:convolution} proves the result for their products.
The monomial shifts and the reversal used in \eqref{A:eq:block} do not
change log-concavity.
\end{proof}

\clearpage
\part*{Part B: Residues plus or minus 2}
\addcontentsline{toc}{part}{Part B: Residues plus or minus 2}
\setcounter{section}{0}
\renewcommand{\thesection}{B.\arabic{section}}
\renewcommand{\theHsection}{\thesection}
\section{Statement and conventions}

Write $W(a,b)$ for the type-$A$ maximal parabolic with top blocks $(a,b)$
and one bottom block $(a+b)$. Only positive coprime block sizes are
considered, so every algebra in this part is Frobenius. The main result is
the following.

\begin{theorem}\label{B:thm:main}
For every $n,m\ge1$, the Ooms multiplicity sequences of
\[
 W\bigl(m(2n+1)+2,\,2n+1\bigr)
 \quad\text{and}\quad
 W\bigl(m(2n+1)-2,\,2n+1\bigr)
\]
are log-concave. Equivalently, this holds for $W(mb\pm2,b)$ whenever
$b\ge3$ is odd and $m\ge1$.
\end{theorem}

These pairs are coprime since $\gcd(mb\pm2,b)=\gcd(2,b)=1$.
When $m=1$ in the minus family, the first block is smaller than the second;
it is retained, not discarded as an exceptional normalization.
Some low-quotient cases overlap the previously studied block-difference-two
families in \cite{MR}. The conclusion here is uniform in both parameters.

For a Laurent polynomial, put $P^\vee(z)=P(z^{-1})$, and write
\[
 I_j=1+z+\cdots+z^{j-1}\quad(j\ge1),\qquad I_0=0.
\]
Orient top arcs right-to-left and bottom arcs left-to-right and let
$\psi(u)=1+\psi(v)$ along an oriented edge $u\to v$.
For two block histograms $H,G$, the reversed spectral polynomial is
\begin{equation}\label{B:eq:spectrum}
 M(H,G)=H^\vee H+H^\vee G+G^\vee G-1.
\end{equation}
The coefficient at $k$ is the multiplicity of the actual eigenvalue $-k$.
This is the standard directed-meander identification, and the known
Frobenius spectral symmetry gives
\begin{equation}\label{B:eq:symmetry}
 M^\vee=zM,\qquad M_k=M_{-1-k};
\end{equation}
see \cite{CHM,MR}. The subtraction of one in \eqref{B:eq:spectrum} is retained
in every identity below. A common monomial translation of $H,G$ has no
effect on $M$.

We use the Euclidean histogram factorization proved in the companion
manuscript~\cite[Proposition~2.2]{Main} and reproduced in Section~\ref{B:app:lifting}. If $A,C$ are
the two histograms of $W(r,b-r)$ and $U=A+C$, then the histograms of
$W(qb+r,b)$ are
\begin{equation}\label{B:eq:lift}
 H_q=A I_{q+1}+C I_q=U I_q+z^qA,\qquad G=z^{-1}U,\qquad q\ge0.
\end{equation}
In particular $q=0$ is an actual smaller meander, not a formal continuation.
The proof below uses no numerical extrapolation and does not assume
log-concavity of all fence rank polynomials or of sums of log-concave
sequences.

\section{Two-parameter histogram and coefficient formulas}

\begin{lemma}[The two seeds]\label{B:lem:seeds}
In a common monomial normalization the seed histograms are
\begin{align}
 W(2,2n-1):\qquad
 A_+&=z^n I_2,& C_+&=1+zI_2I_{n-1},&
 U_+&=1+zI_2I_n,\label{B:eq:plus-seed}\\
 W(2n-1,2):\qquad
 A_-&=z^{n+1}+zI_2I_{n-1},&
 C_-&=I_2,&
 U_-&=I_2I_n+z^{n+1}.\label{B:eq:minus-seed}
\end{align}
\end{lemma}
\begin{proof}
In a polynomial normalization the two subtractive histogram moves are
\[
 (A,C)\longmapsto(z(A+C),C),\qquad
 (A,C)\longmapsto(zA,A+C).
\]
Away from the initial pair they follow from \eqref{B:eq:lift}; the second
is a common translate of the opposed move. The two moves out of $W(1,1)$
are checked directly from its two- and three-vertex meanders. Starting at $(A,C)=(z,1)$ for $W(1,1)$, one left move
followed by $n-1$ right moves gives \eqref{B:eq:plus-seed}. One right move
followed by $n-1$ left moves gives \eqref{B:eq:minus-seed}.
Both formulas also hold at $n=1$, with $I_0=0$.
\end{proof}

Consequently the parent histograms are
\begin{align}
 H_+&=U_+I_m+z^mA_+,&G_+&=z^{-1}U_+,\label{B:eq:plus-parent}\\
 H_-&=U_-I_{m-1}+z^{m-1}A_-,&
 G_-&=z^{-1}U_-.\label{B:eq:minus-parent}
\end{align}
Define the normalized spectral polynomials
\begin{equation}\label{B:eq:F}
 F^+_{n,m}=z^{n+m+2}M(H_+,G_+),\qquad
 F^-_{n,m}=z^{n+m+1}M(H_-,G_-).
\end{equation}
Put
\[
 N_+=n+m+1,\qquad N_-=n+m.
\]
These polynomials have nonnegative coefficients, constant term one, and
symmetry about $N_\pm+\tfrac12$. Thus their degree is $2N_\pm+1$.
The constant term and polynomial normalization also follow from the next
explicit identity. Let
\[
 L_k^\pm=[z^k]F^\pm_{n,m},\qquad 0\le k\le N_\pm.
\]

For a polynomial $Q$, write
\[
 c_Q(j)=[z^j]\frac{Q(z)}{(1-z)^4},
 \qquad
 C_Q(j)=\sum_{\ell\le j}[z^\ell]Q,
\]
where coefficients at negative degrees are zero. We use the following small
polynomials:
\begin{center}
\small
\begin{tabular}{ccl}
\toprule
Sign & Symbol & Polynomial\\
\midrule
$+$ & $B_+$ & $1+z^2+z^3+z^5$\\
$+$ & $G_{A,+}$ & $2z+z^2+2z^3+z^4+2z^5$\\
$+$ & $G_{B,+}$ & $z+2z^2+2z^3+2z^4+z^5$\\
$+$ & $J_+$ & $z^2+z^3+z^4+z^5$\\
\midrule
$-$ & $B_-$ & $1+z+z^2+z^3$\\
$-$ & $G_{A,-}$ & $2+z+2z^2+z^3+2z^4$\\
$-$ & $G_{B,-}$ & $1+2z+2z^2+2z^3+z^4$\\
$-$ & $J_-$ & $1+z^2+z^3+z^5$\\
\bottomrule
\end{tabular}
\end{center}

\begin{proposition}[Exact rising-half formula]\label{B:prop:coefficients}
For either sign and $0\le k\le N_\pm$,
\begin{equation}\label{B:eq:half-coeff}
\boxed{
\begin{aligned}
 L_k^\pm={}&c_{B_\pm}(k)
 -c_{G_{A,\pm}}(k-n)-c_{G_{B,\pm}}(k-m)\\
 &+c_{J_\pm}(k-2n)+c_{J_\pm}(k-2m)
 +\delta_{k,N_\pm}.
\end{aligned}}
\end{equation}
Every $c_Q$ is explicit:
\begin{equation}\label{B:eq:binomial-rule}
 c_Q(j)=\sum_{\ell\le j}[z^\ell]Q\binom{j-\ell+3}{3}.
\end{equation}
\end{proposition}
\begin{proof}
Set $X=z^n,Y=z^m$ and substitute \eqref{B:eq:plus-parent} and
\eqref{B:eq:minus-parent} in \eqref{B:eq:spectrum}. Multiplication by
$(1-z)^4$ gives, for the plus case,
\begin{align}
 &(1-z)^4F^+_{n,m}\label{B:eq:numplus}\\
 &\quad=B_+-XG_{A,+}-YG_{B,+}+(X^2+Y^2)J_+
 +XY B_{c,+}\nonumber\\
 &\qquad-X^2Y\,zG_{B,+}-XY^2\,zG_{A,+}+X^2Y^2z^2B_+,\nonumber\\
 & B_{c,+}=z+2z^2+5z^3+5z^4+2z^5+z^6.\nonumber
\end{align}
For the minus case it gives
\begin{align}
 &(1-z)^4F^-_{n,m}\label{B:eq:numminus}\\
 &\quad=B_--XG_{A,-}-YG_{B,-}+(X^2+Y^2)J_-
 +XY B_{c,-}\nonumber\\
 &\qquad-X^2Y G_{B,+}-XY^2G_{A,+}+X^2Y^2J_+,\nonumber\\
 & B_{c,-}=1+2z+5z^2+5z^3+2z^4+z^5.\nonumber
\end{align}
These are polynomial identities in the independent symbols $z,X,Y$
after clearing the geometric-series denominators.

In \eqref{B:eq:numplus}, all terms involving $X^2Y$, $XY^2$, or $X^2Y^2$
begin strictly after $N_+$. The $XY$ term first appears at $N_+$ with
coefficient one. The same assertions hold for \eqref{B:eq:numminus} with
$N_-$. Taking coefficients through these degrees proves
\eqref{B:eq:half-coeff}. Formula \eqref{B:eq:binomial-rule} is the binomial
series for $(1-z)^{-4}$.
\end{proof}

\begin{remark}[A positive identity, not an LC closure argument]
The plus polynomial also has the exact factorization
\[
\begin{aligned}
 F^+_{n,m}={}&
 zI_4I_nI_{n+1}I_mI_{m+1}
 +zI_3I_nI_{n+2m+1}
 +I_{n+m+2}I_{n+m+3}.
\end{aligned}
\]
The three summands have the same symmetry center. This proves unimodality
directly, but the proof of log-concavity below does not assume that their
sum preserves log-concavity.
\end{remark}

\section{The final half-step}

\begin{lemma}\label{B:lem:last-step}
For both families and all $n,m\ge1$,
\[
 L_{N_\pm}^\pm\ge L_{N_\pm-1}^\pm.
\]
\end{lemma}
\begin{proof}
For either seed write $U=A+C$, $K=zA+C$, $R=U^\vee U$, $E=K^\vee U$.
If $\Sigma=A^\vee A+A^\vee C+C^\vee C-1$ is the seed spectrum, its
symmetry gives, with $D=A^\vee A+C^\vee C$ and $Y_0=A^\vee C$,
\[
 Y_0^\vee-zY_0=(z-1)(D-1).
\]
It follows that
\[
 K^\vee K=(z+z^{-1}-1)R+(2-z-z^{-1}).
\]
Using $(1-z)H_q=U-z^qK$ and $(1-z)G=(z^{-1}-1)U$ in
\eqref{B:eq:spectrum} now gives
\begin{equation}\label{B:eq:three-cluster}
 \LC M_q=(1+z^{-1})R-z^{-q-1}E-z^qE^\vee,
 \qquad \LC=(1-z)(1-z^{-1}).
\end{equation}
Here $q=m$ for the plus family and $q=m-1$ for the minus family.

The following interval decompositions show that the coefficients of
$T=A^\vee U$ are nonincreasing for degrees $j\ge0$.
For the plus seed,
\[
 A=z^nI_2,\qquad U=I_{n+2}+z^2I_{n-1}.
\]
The midpoint of the $A$ interval is no smaller than either midpoint of
the $U$ intervals. For the minus seed and $n\ge2$,
\[
 A=zI_{n+1}+z^2I_{n-2},\qquad
 U=I_{n+2}+zI_{n-1},
\]
and again every $A$ midpoint is no smaller than every $U$ midpoint.
An empty interval contributes zero. At $n=1$ in the minus case,
$A=z^2,U=I_3$, so the same conclusion is immediate.

Indeed, the correlation of two interval indicators is a trapezoid centered
at the difference of their midpoints. If that midpoint is nonpositive,
its coefficients cannot increase at nonnegative integer degrees.
Also $U$ is a sum of nested interval indicators. Its autocorrelation
$R$ is therefore nonincreasing at nonnegative degrees: each ordered
interval-pair overlap is maximal at zero translation and cannot increase
when the shift moves in a fixed direction away from zero.

Since
\[
 E=R+(z^{-1}-1)T,
\]
we obtain $0\le E_j\le R_j$ for $j\ge0$. At degree $-1$,
\eqref{B:eq:three-cluster} gives
\[
 [z^{-1}]\LC M_q=R_0+R_1-E_q-E_{q+1}\ge0.
\]
By $[z^{-1}]M_q=[z^0]M_q$, the left side equals
$[z^{-1}]M_q-[z^{-2}]M_q$, exactly the last half-step in the normalized polynomial.
\end{proof}

\section{Third differences after the head}

Put $p=\min(n,m)$ and
\[
 K_+=p+1,\qquad K_-=p.
\]
Here $K_\pm$ is an index, not the seed polynomial $K$ used above.
For a rising-half sequence, put
\[
 d_k=L_k-L_{k-1},\qquad
 w_k=L_{k+1}-2L_k+L_{k-1}.
\]

\begin{lemma}[Fixed cumulative packets]\label{B:lem:packets}
For either sign, either $G=G_{A,\pm}$ or $G=G_{B,\pm}$, every $q\ge1$,
and every integer $t$, put
\[
 P_{G,q}(t)=C_G(t-q)-C_{J_\pm}(t-2q).
\]
Then $P_{G,q}(t)\ge0$ always. In the plus case it is at least four when
$t\ge q+3$; in the minus case it is at least four when $t\ge q+2$.
\end{lemma}
\begin{proof}
Write $u=t-q$. Since $q\ge1$,
\[
 P_{G,q}(t)\ge C_G(u)-C_{J_\pm}(u-1).
\]
The two lower bounds on the right, in the order $G_A,G_B$, are
\[
\begin{array}{c|rrrrrrr}
 u&0&1&2&3&4&5&u\ge6\\ \hline
 +&(0,0)&(2,1)&(3,3)&(4,4)&(4,5)&(5,5)&(4,4)\\
 -&(2,1)&(2,2)&(4,4)&(4,5)&(5,5)&(5,5)&(4,4).
\end{array}
\]
For $u<0$ they are zero. The table follows by summing the displayed
polynomials, which have degree at most five; the last column is constant
because the cumulative sums have reached their total masses.
\end{proof}

\begin{proposition}[Nonincreasing second differences]\label{B:prop:third}
For either sign,
\begin{equation}\label{B:eq:third}
 w_{k+1}\le w_k\qquad(K_\pm\le k\le N_\pm-2).
\end{equation}
An empty range imposes no condition.
\end{proposition}
\begin{proof}
Taking third differences of \eqref{B:eq:half-coeff}, and putting $t=k+2$,
gives
\[
\begin{aligned}
 w_{k+1}-w_k
 ={}&C_{B_\pm}(t)
 -P_{G_{A,\pm},n}(t)-P_{G_{B,\pm},m}(t)
 +\mathbf1_{\{t=N_\pm\}}.
\end{aligned}
\]
The free cumulative sum is at most four. In the stated range the packet
associated with the smaller parameter is at least four by
Lemma~\ref{B:lem:packets}, while the other is nonnegative. This proves the
claim unless $t=N_\pm$.

Let the larger parameter be $q\ge p$.
At $t=N_+=p+q+1$, its packet is
\[
 C_G(p+1)-C_{J_+}(p-q+1)\ge3,
\]
because $J_+$ starts in degree two and $p-q+1\le1$.
At $t=N_-=p+q$, its packet is
\[
 C_G(p)-C_{J_-}(p-q)\ge3-1=2.
\]
Thus the extra one at the central degree is absorbed in both cases.
\end{proof}

\begin{lemma}[A head-to-tail criterion]\label{B:lem:tail}
Suppose $L_0,\ldots,L_N$ is the nonnegative half of a symmetric sequence
with repeated central pair. Let $1\le K\le N-1$. Assume
\[
 D_k:=L_k^2-L_{k-1}L_{k+1}\ge0\quad(1\le k\le K),
\]
\[
 d_{K+1}\ge0,\qquad d_N\ge0,
 \qquad w_{k+1}\le w_k\quad(K\le k\le N-2).
\]
Then the full symmetric sequence is log-concave.
\end{lemma}
\begin{proof}
The slope sequence $d_{K+1},\ldots,d_N$ is ordinarily concave, since its
successive differences are the nonincreasing $w_{K+1},\ldots,w_{N-1}$.
It is nonnegative because both end values are nonnegative. Direct expansion
gives
\begin{equation}\label{B:eq:defect-propagation}
 D_{k+1}-D_k
 =w_kd_{k+1}+L_k(w_k-w_{k+1}).
\end{equation}
As long as $w_k\ge0$, this is nonnegative. Once $w_k<0$, all subsequent
$w$ are nonpositive, so ordinary concavity gives all remaining internal
log-concavity inequalities. The two central inequalities reduce to
$d_N\ge0$, and reflection gives the rest. The case $K=N-1$ has no
tail and follows directly from the same central check.
\end{proof}

\section{Symbolic verification of the initial segment}

\subsection{Both parameters at least three}

Assume $p=\min(n,m)\ge3$ and put $s=p-3\ge0$.
Before either translated subtraction begins, the coefficient formulas are
\[
 f_+(k)=\frac23k^3-k^2+\frac{19}3k-3\quad(k\ge2),
 \qquad
 f_-(k)=\frac23k^3+k^2+\frac73k+1\quad(k\ge0).
\]
In the plus case the initial coefficients at degrees zero and one are
$1,4$. The first two defects are $5,21$. For $k\ge3$ the free defect is
\[
 D(f_+;k)=\frac23(2k^4-4k^3+k^2+k+63)>0.
\]
Indeed, on putting $v=k-3\ge0$, three times this expression is
\[
 4v^4+40v^3+146v^2+230v+258.
\]
In the minus case the free defects are
\[
 D(f_-;k)=\frac23(2k^4+4k^3+k^2-k+9)>0\qquad(k\ge1).
\]
Putting $v=k-1$ turns three times this expression into
\[
 4v^4+24v^3+50v^2+42v+30.
\]

Only two transition defects remain in either sign.
For compactness, let a tuple $(a,b,c,d)$ stand for the polynomial
\[
 \Pi(a,b,c,d;s)=4s^4+as^3+bs^2+cs+d.
\]
The following tables give \emph{three times} the indicated defect.
Every printed coefficient is nonnegative and every constant term is positive.

\begin{center}
\small
\begin{tabular}{lll}
\toprule
Plus-family parameter chamber & $3D_p^+$ & $3D_{p+1}^+$\\
\midrule
$n=m=p$ & $(46,173,323,357)$ & $(74,389,937,948)$ \\
$n=p,\ m=p+1$ & $(44,164,292,324)$ & $(68,356,844,852)$ \\
$n=p,\ m\ge p+2$ & $(44,164,292,324)$ & $(66,341,789,777)$ \\
$m=p,\ n=p+1$ & $(42,155,261,291)$ & $(68,368,916,987)$ \\
$m=p,\ n\ge p+2$ & $(42,155,261,291)$ & $(64,338,806,837)$ \\
\bottomrule
\end{tabular}
\end{center}

\begin{center}
\small
\begin{tabular}{lll}
\toprule
Minus-family parameter chamber & $3D_{p-1}^-$ & $3D_p^-$\\
\midrule
$n=m=p$ & $(46,173,287,195)$ & $(74,389,829,642)$ \\
$n=p,\ m=p+1$ & $(44,164,268,180)$ & $(68,356,772,612)$ \\
$n=p,\ m\ge p+2$ & $(44,164,268,180)$ & $(66,341,729,567)$ \\
$m=p,\ n=p+1$ & $(42,155,249,165)$ & $(68,368,844,723)$ \\
$m=p,\ n\ge p+2$ & $(42,155,249,165)$ & $(64,338,758,633)$ \\
\bottomrule
\end{tabular}
\end{center}

Here is an explicit way to check all entries without any parameter search.
Through the relevant head, neither $z^{2n}J_\pm$, $z^{2m}J_\pm$, nor the
central term contributes. If $n=p,m=p+\delta$, substitute
\[
 L_{p+j}^\pm=f_\pm(p+j)
   -c_{G_{A,\pm}}(j)-c_{G_{B,\pm}}(j-\delta)
\]
into $D_k=L_k^2-L_{k-1}L_{k+1}$.
If $m=p,n=p+\delta$, interchange the two offsets.
For the plus head,
\[
 (c_{G_{A,+}}(0),c_{G_{A,+}}(1),c_{G_{A,+}}(2))=(0,2,9),
\]
\[
 (c_{G_{B,+}}(0),c_{G_{B,+}}(1),c_{G_{B,+}}(2))=(0,1,6).
\]
For the minus head, use
$c_{G_{A,-}}(j)=c_{G_{A,+}}(j+1)$ and the analogous $B$ identity.
Negative arguments are zero. Thus $\delta=0,1,$ or $\delta\ge2$ exhausts
the possibilities; there is no unlisted large-gap case.

The initial slope of the tail is positive as well. In the plus case,
\[
 d_{p+2}^+\ge2p^2+4p-4>0,
\]
and in the minus case,
\[
 d_{p+1}^-\ge2p^2+4p-8>0.
\]
To obtain these bounds, the free first differences are
$2k^2-4k+8$ and $2k^2+2$, respectively. At the indicated step the two
subtractions can reduce the slope by at most
$(9-2)+(6-1)=12$.

\subsection{One parameter equal to one or two}

When $p\le2$ and the other parameter is at least four, the relevant head
is independent of that other parameter. Its translated subtraction starts
strictly after the last displayed coefficient. Direct coefficient extraction
from \eqref{B:eq:half-coeff} gives the following complete table.
The defect column lists $D_1,\ldots,D_{K_\pm}$, and the final column is
$d_{K_\pm+1}$.

\begin{center}
\small
\begin{tabular}{clllr}
\toprule
Sign & Parameters & $(L_0,\ldots,L_{K_\pm+1})$ & Head defects & Step\\
\midrule
$+$ & $n=1,\ m\ge4$ & $(1,4,9,16)$ & $(7,17)$ & $7$ \\
$+$ & $m=1,\ n\ge4$ & $(1,4,10,19)$ & $(6,24)$ & $9$ \\
$+$ & $n=2,\ m\ge4$ & $(1,4,11,23,40)$ & $(5,29,89)$ & $17$ \\
$+$ & $m=2,\ n\ge4$ & $(1,4,11,24,43)$ & $(5,25,103)$ & $19$ \\
$-$ & $n=1,\ m\ge4$ & $(1,3,7)$ & $(2)$ & $4$ \\
$-$ & $m=1,\ n\ge4$ & $(1,4,10)$ & $(6)$ & $6$ \\
$-$ & $n=2,\ m\ge4$ & $(1,5,13,26)$ & $(12,39)$ & $13$ \\
$-$ & $m=2,\ n\ge4$ & $(1,5,14,29)$ & $(11,51)$ & $15$ \\
\bottomrule
\end{tabular}
\end{center}

The remaining pairs have $n,m\le3$ and $\min(n,m)\le2$. Their complete
half-sequences are printed in Section~\ref{B:app:bases}. That section is
a finite exact calculation, not a claim that bounded testing settles an
unbounded family.

\begin{proof}[Proof of Theorem~\ref{B:thm:main}]
For $p\ge3$ the positive quartic tables and the free-head calculations prove
all internal defects through $K_\pm$. The initial tail slope is positive.
For $p\le2$ with the larger parameter at least four, the constant-head table
proves the same assertions. Lemma~\ref{B:lem:last-step} supplies the final
half-step, and Proposition~\ref{B:prop:third} supplies the nonincreasing
second differences after the head. Lemma~\ref{B:lem:tail} therefore proves
log-concavity for all these parameters. Section~\ref{B:app:bases} covers
every remaining pair. Reversing the spectral convention does not change
log-concavity.
\end{proof}

\section{The finite base rows}\label{B:app:bases}

Each word below is the entire rising half. The minimum is taken over its
internal log-concavity defects; the last column is the final half-step.
All are positive. The rows follow from \eqref{B:eq:binomial-rule} and
\eqref{B:eq:half-coeff}.

\begin{center}
\small
\begin{longtable}{cclrr}
\toprule
Sign & $(n,m)$ & Rising half & Minimum defect & Final step\\
\midrule
\endhead
$+$ & $(1,1)$ & $(1,4,8,11)$ & $8$ & $3$ \\
$+$ & $(1,2)$ & $(1,4,9,15,19)$ & $7$ & $4$ \\
$+$ & $(1,3)$ & $(1,4,9,16,23,28)$ & $7$ & $5$ \\
$+$ & $(2,1)$ & $(1,4,10,17,22)$ & $6$ & $5$ \\
$+$ & $(2,2)$ & $(1,4,11,22,34,42)$ & $5$ & $8$ \\
$+$ & $(2,3)$ & $(1,4,11,23,39,55,66)$ & $5$ & $11$ \\
$+$ & $(3,1)$ & $(1,4,10,19,28,34)$ & $6$ & $6$ \\
$+$ & $(3,2)$ & $(1,4,11,24,41,58,69)$ & $5$ & $11$ \\
$-$ & $(1,1)$ & $(1,2,3)$ & $1$ & $1$ \\
$-$ & $(1,2)$ & $(1,3,6,8)$ & $3$ & $2$ \\
$-$ & $(1,3)$ & $(1,3,7,12,16)$ & $2$ & $4$ \\
$-$ & $(2,1)$ & $(1,4,8,11)$ & $8$ & $3$ \\
$-$ & $(2,2)$ & $(1,5,12,20,26)$ & $13$ & $6$ \\
$-$ & $(2,3)$ & $(1,5,13,25,38,47)$ & $12$ & $9$ \\
$-$ & $(3,1)$ & $(1,4,10,17,22)$ & $6$ & $5$ \\
$-$ & $(3,2)$ & $(1,5,14,27,41,50)$ & $11$ & $9$ \\
\bottomrule
\end{longtable}
\end{center}

\section{Euclidean histogram factorization}\label{B:app:lifting}

For clarity, this appendix records the combinatorial input behind
\eqref{B:eq:lift}. Let $u_1,\ldots,u_b$ be a potential of the seed
$W(r,b-r)$, and let $\tau$ be its top reflection:
\[
 \tau(s)=
 \begin{cases}
 r+1-s,&s\le r,\\
 b+r+1-s,&s>r.
 \end{cases}
\]
For $q\ge1$, set $a=qb+r$ and $Q_s=q$ for $s\le r$, $Q_s=q-1$ otherwise.
The parent potential on the residue chain is
\[
 \psi(s+jb)=
 \min\{u_s+2j,\ u_{\tau(s)}+1+2(Q_s-j)\},
 \qquad 0\le j\le Q_s.
\]
The small-block values are
\[
 \psi(a+t)=u_{b+1-t}-1.
\]

To verify the formula, first take a two-element top orbit
$s<t=\tau(s)$, with $u_s=u,u_t=u+1$ and $Q_s=Q_t=Q$. Its two sequences are
\[
 u+\min(2j,2Q-2j+2),\qquad
 u+1+\min(2j,2Q-2j).
\]
The top mate of $s+jb$ is $t+(Q-j)b$. If $2j\le Q$, their values are
$u+2j,u+2j+1$ in increasing vertex order; if $2j>Q$, their values are
$u+2Q-2j+2,u+2Q-2j+1$ in the opposite order.
For $1\le j\le Q$, its bottom mate is $t+(Q+1-j)b$.
If $2j\le Q+1$, the values are $u+2j,u+2j-1$ in increasing vertex
order; otherwise they are $u+2Q-2j+2,u+2Q-2j+3$ in the opposite order.
These are exactly the top and bottom potential relations.

At a fixed seed residue, the formula is
$u_s+\min(2j,2Q_s-2j+1)$. The same check uses indices
$j,Q_s-j$ for top mates and $j,Q_s+1-j$ for bottom mates; equal indices
give a fixed reflection vertex, not an edge.
The boundary bottom edges join values $u_s,u_s-1$.
The reversed small-block top edges are precisely seed bottom edges.
Thus every edge relation is verified, and uniqueness on the Frobenius path
identifies the potential.

For a two-element orbit, the two sequences together contribute
\[
 z^u(1+z)I_{Q+1}.
\]
Indeed the endpoints $u,u+Q+1$ occur once and all intervening heights occur
twice: when $Q$ is even the two parity branches provide alternating even
and odd heights, and when $Q$ is odd their two terminal single occurrences
interchange. A fixed orbit contributes $z^{u_s}I_{Q_s+1}$, since its two
branches give each height once. Summing the orbit contributions multiplies
the first seed block by $I_{q+1}$ and the second by $I_q$, proving
$H_q=A I_{q+1}+C I_q$; the small-block formula gives $G=z^{-1}(A+C)$.

For $q=0$, assign $\psi(i)=u_i$ on $1\le i\le r$, and
$\psi(r+t)=u_{b+1-t}-1$ on the second block.
First-block top edges remain seed top edges, and second-block top edges
become seed bottom edges under reversal.
Bottom edges incident to the first block join $u_i,u_i-1$.
For a bottom edge contained in the second block, its offsets satisfy
$t+t'=b-r+1$; the reversed seed indices $b+1-t,b+1-t'$ lie in the
second seed top block and sum to $b+r+1$. Its seed top relation is therefore
exactly the required parent bottom relation. This proves the genuine
$q=0$ case used in \eqref{B:eq:minus-parent}.

\clearpage
\part*{Part C: Residues plus or minus 3}
\addcontentsline{toc}{part}{Part C: Residues plus or minus 3}
\setcounter{section}{0}
\renewcommand{\thesection}{C.\arabic{section}}
\renewcommand{\theHsection}{\thesection}
\section{Statement and conventions}

All algebras in this part are Frobenius. For positive coprime block sizes,
$W(a,b)$ denotes the type-$A$ maximal parabolic with top composition
$(a,b)$ and one bottom block $(a+b)$. We use the reversed multiplicity
polynomial, whose coefficient at $k$ is the multiplicity of the actual
eigenvalue $-k$.

\begin{theorem}\label{C:thm:main}
Let $n,m\ge1$ and $r\in\{1,2\}$. Put $b=3n+r$. Both
\[
 W(mb+3,b)\qquad\text{and}\qquad W(mb-3,b)
\]
have log-concave Ooms multiplicity sequences.
\end{theorem}

The coprimality assertion is automatic:
$\gcd(mb\pm3,b)=\gcd(3,b)=1$. The minus family at $m=1$ is included,
even though its first block is smaller than its second. The case $b=2$
follows from the block-size-two case in the main manuscript.

Write $F^\vee(z)=F(z^{-1})$ and
\[
 I_j=1+z+\cdots+z^{j-1}\quad(j\ge1),\qquad I_0=0.
\]
The potential decreases by one along top edges oriented right-to-left and
bottom edges oriented left-to-right. Thus a top pair $x<y$ satisfies
$\psi(y)=\psi(x)+1$, and a bottom pair satisfies
$\psi(x)=\psi(y)+1$. For block histograms $H,G$, the standard
principal-element calculation gives
\begin{equation}\label{C:eq:spectrum}
 M(H,G)=H^\vee H+H^\vee G+G^\vee G-1.
\end{equation}
The last term is the trace-zero correction. Frobenius spectral symmetry is
\begin{equation}\label{C:eq:symmetry}
 M^\vee=zM,\qquad M_k=M_{-1-k};
\end{equation}
The same results give an unbroken integer support, so every coefficient
between the endpoints is positive; see \cite{CHM,MR}. A common translation of both histograms does not affect
\eqref{C:eq:spectrum}.

The Euclidean histogram formula is
\begin{equation}\label{C:eq:lift}
 H_q=(A+C)I_q+z^qA,\qquad G=z^{-1}(A+C),\qquad q\ge0,
\end{equation}
where $A,C$ are the seed histograms for $W(u,b-u)$, and $H_q,G$ are
those for $W(qb+u,b)$. Section~\ref{C:app:lift} records its edge and
multiset verification, including the genuine $q=0$ case.

\section{Exact seed and parent polynomials}

\begin{lemma}\label{C:lem:seed}
Put $b=3n+r$, where $n\ge1$ and $r=1,2$. In a common polynomial
normalization, the plus seed $W(3,b-3)$ has
\begin{equation}\label{C:eq:seedplus}
 A_+=z^n I_3,\qquad C_+=I_r+zI_3I_{n-1},\qquad
 U_+=I_r+zI_3I_n.
\end{equation}
The minus seed $W(b-3,3)$ has
\begin{align}
 C_-&=I_3,\label{C:eq:seedminus}\\
 A_-&=\begin{cases}
 z^{n+2}+zI_3I_{n-1},&r=1,\\
 z^{n+1}I_2+zI_3I_{n-1},&r=2,
 \end{cases}\nonumber\\
 U_-&=\begin{cases}
 z^{n+2}+I_3I_n,&r=1,\\
 z^{n+1}I_2+I_3I_n,&r=2.
 \end{cases}\nonumber
\end{align}
\end{lemma}
\begin{proof}
In a polynomial normalization, the subtractive histogram moves are
\[
 (A,C)\longmapsto(z(A+C),C),\qquad
 (A,C)\longmapsto(zA,A+C).
\]
The first follows by comparing $H_{q+1}=zH_q+U$ with $G=z^{-1}U$
in \eqref{C:eq:lift}. Reflecting the vertex line replaces a histogram pair
by $(C^\vee,A^\vee)$, up to a common shift. Conjugating the first move
by this reflection gives $(A,z^{-1}(A+C))$; a common multiplication by
$z$ gives the second displayed move. Direct edge checking
at the four- and five-vertex bases gives
\[
\begin{array}{c|cc}
 &A&C\\\hline
 W(3,1)&zI_3&1\\
 W(3,2)&zI_3&I_2\\
 W(1,3)&z^3&I_3\\
 W(2,3)&z^2I_2&I_3.
\end{array}
\]
For the plus seed, apply $n-1$ right moves to the first or second row.
The first histogram becomes $z^nI_3$, and the second becomes
$I_r+zI_3I_{n-1}$. For the minus seed, apply $n-1$ left moves to the
third or fourth row. The second histogram stays $I_3$, while the first
is multiplied by $z^{n-1}$ and gains $zI_3I_{n-1}$. Summing the two
histograms gives the stated formulas for $U_\pm$.
\end{proof}

The parent histograms are consequently
\begin{align}
 H_+&=U_+I_m+z^mA_+,&G_+&=z^{-1}U_+,\label{C:eq:parents}\\
 H_-&=U_-I_{m-1}+z^{m-1}A_-,&G_-&=z^{-1}U_- .\nonumber
\end{align}
Set
\begin{equation}\label{C:eq:normal}
 F^+_{r,n,m}=z^{n+m+3}M(H_+,G_+),\qquad
 F^-_{r,n,m}=z^{n+m+2}M(H_-,G_-),
\end{equation}
and put
\[
 N_+=n+m+2,\qquad N_-=n+m+1.
\]
The normalized spectrum has the symmetric coefficient word
\[
 L_0,\ldots,L_N,L_N,\ldots,L_0,\qquad N=N_\pm.
\]
The next identity also shows that the normalization has constant term one
and no negative powers; symmetry then gives degree $2N+1$.

Define fixed polynomials
\begin{align*}
 J&=z^2-z+1,\\
 D_1&=(1+z+z^2)(1+z^2+z^3),\\
 D_2&=(1+z+z^2)(1+z+z^3),\\
 T&=(1+z^2)(2+2z+z^2+2z^3+2z^4),\\
 C&=1+3z+5z^2+9z^3+9z^4+5z^5+3z^6+z^7.
\end{align*}
Here the letter $C$ in the displayed list denotes a fixed numerator
polynomial, not a seed block histogram. The subscript $\bar r$ means
$3-r$.

\begin{proposition}[Four rational identities]\label{C:prop:numerator}
With $X=z^n$, $Y=z^m$, one has
\begin{align}
(1-z)^4F^+_{r,n,m}
={}&JD_r-zXT-zY(1+z)D_r+z^2X^2D_{\bar r}+z^2Y^2D_r
\label{C:eq:fullplus}\\
 &+zXYC-z^2X^2Y(1+z)D_{\bar r}-z^2XY^2T
   +z^2X^2Y^2JD_{\bar r},\nonumber\\
(1-z)^4F^-_{r,n,m}
={}&D_r-XT-Y(1+z)D_r+X^2JD_{\bar r}+Y^2JD_r
\label{C:eq:fullminus}\\
 &+XYC-zX^2Y(1+z)D_{\bar r}-zXY^2T+z^2X^2Y^2D_{\bar r}.
\nonumber
\end{align}
These are identities in the independent symbols $z,X,Y$ after the
geometric-series denominators are cleared.
\end{proposition}
\begin{proof}
In \eqref{C:eq:parents}, replace $I_n$ by $(1-X)/(1-z)$ and $I_m$ by
$(1-Y)/(1-z)$; replace $I_{n-1},I_{m-1}$ by
$(1-X/z)/(1-z),(1-Y/z)/(1-z)$, respectively. Compute reversals by
$(z,X,Y)\mapsto(z^{-1},X^{-1},Y^{-1})$, substitute into
\eqref{C:eq:spectrum}, multiply by the normalizing monomial in
\eqref{C:eq:normal} and by $(1-z)^4$, and collect the nine monomials
$X^iY^j$, $0\le i,j\le2$. Their coefficients are exactly the nine
terms printed in each formula. This describes a finite rational identity,
not an inference from any bounded parameter range.
\end{proof}

For a polynomial $Q$, put
\[
 c_Q(j)=\sum_{\ell\le j}[z^\ell]Q\binom{j-\ell+3}{3},\qquad
 C_Q(j)=\sum_{\ell\le j}[z^\ell]Q,
\]
with either expression zero for $j<0$. Define the five packets by
\[
\begin{array}{c|ccccc}
 &B&G_A&G_B&J_A&J_B\\\hline
 +&JD_r&zT&z(1+z)D_r&z^2D_{\bar r}&z^2D_r\\
 -&D_r&T&(1+z)D_r&JD_{\bar r}&JD_r.
\end{array}
\]
All five have nonnegative coefficients, with masses $9,18,18,9,9$.
Their coefficient lists are recorded in Section~\ref{C:app:packets}.

\begin{corollary}[Exact half-sequence]\label{C:cor:half}
For $0\le k\le N=N_\pm$,
\begin{equation}\label{C:eq:half}
\boxed{
\begin{aligned}
 L_k={}&c_B(k)-c_{G_A}(k-n)-c_{G_B}(k-m)\\
 &+c_{J_A}(k-2n)+c_{J_B}(k-2m)
 +\delta_{k,N-1}+7\delta_{k,N}.
\end{aligned}}
\end{equation}
\end{corollary}
\begin{proof}
In \eqref{C:eq:fullplus}, the $XY$ term starts in degree $N_+-1$ with
coefficient one; after division by $(1-z)^4$, its coefficient in degree
$N_+$ is $4+3=7$. Every term involving $X^2Y,XY^2,X^2Y^2$ starts
strictly beyond $N_+$. The same statements hold for
\eqref{C:eq:fullminus} with $N_-$. Coefficient extraction proves the result.
\end{proof}

\section{The central pair is checked directly}

\begin{lemma}\label{C:lem:central}
For all four families and all $n,m\ge1$, $L_N\ge L_{N-1}$.
\end{lemma}
\begin{proof}
For either seed put $R=U^\vee U$, $E=(zA+C)^\vee U$ and
$\LC=(1-z)(1-z^{-1})$. Seed spectral symmetry and a direct expansion of
\eqref{C:eq:spectrum} give
\begin{equation}\label{C:eq:curvature}
 \LC M_q=(1+z^{-1})R-z^{-q-1}E-z^qE^\vee,
\end{equation}
where $q=m$ in the plus case and $q=m-1$ in the minus case. One explicit
way to check the cancellation is to put $K=zA+C$ and verify
\[
 K^\vee K=(z+z^{-1}-1)U^\vee U+(2-z-z^{-1});
\]
substitute $(1-z)H_q=U-z^qK$ and $(1-z)G=(z^{-1}-1)U$.
The scalar term here cancels the trace-zero term in the spectrum.

The following are the superlevel interval layers of $A$ and $U$ for
$n\ge4$. A list of intervals denotes the sum of their indicator
polynomials.
\begin{center}\small
\begin{tabular}{cll}\toprule
Family&Layers of $A$&Layers of $U$\\\midrule
$+,1$&$[n,n+2]$&$[0,n+2], [2,n+1], [3,n]$\\
$+,2$&$[n,n+2]$&$[0,n+2], [1,n+1], [3,n]$\\
$-,1$&$[1,n+2], [2,n], [3,n-1]$&$[0,n+2], [1,n], [2,n-1]$\\
$-,2$&$[1,n+2], [2,n+1], [3,n-1]$&$[0,n+2], [1,n+1], [2,n-1]$\\
\bottomrule\end{tabular}
\end{center}
For each family, every midpoint in the $A$ column is no smaller than every
midpoint in the $U$ column. Section~\ref{C:app:layers} lists $n=1,2,3$
and checks the same property.

The correlation of two interval indicators is a trapezoid centered at the
difference of their midpoints. Thus $T=A^\vee U$ is nonincreasing at
nonnegative degrees. Also $U$ is unimodal, and its autocorrelation $R$ is
nonincreasing at nonnegative degrees: expand into correlations of its nested
superlevel intervals; each overlap is maximal at zero and cannot increase
in either fixed direction. These assertions concern monotonicity, not
log-concavity of arbitrary sums.

Since
\[
 E=R+(z^{-1}-1)T,
\]
we have $0\le E_j\le R_j$ for $j\ge0$. Spectral symmetry gives
$M_{-1}=M_0$, so the central half-step is
\[
 L_N-L_{N-1}=[z^0]\LC M_q=R_0+R_1-E_q-E_{q+1}\ge0.
\]
This also applies to $q=0$.
\end{proof}

\section{The second differences decrease after a short head}

Write
\[
 d_k=L_k-L_{k-1},\qquad w_k=L_{k+1}-2L_k+L_{k-1}.
\]
The following elementary lemma is the shape-preservation step used in the
proof. It avoids any general assertion about sums of log-concave sequences.

\begin{lemma}[Tail propagation]\label{C:lem:propagation}
Let $L_0,\ldots,L_N$ be positive and let $1\le K\le N-2$.
Suppose the log-concavity defects $D_k=L_k^2-L_{k-1}L_{k+1}$ are
nonnegative for $1\le k\le K$, the two slopes $d_{K+1},d_N$ are
nonnegative, and $w_{k+1}\le w_k$ for $K\le k\le N-2$.
Then the whole half-sequence is log-concave and $L_N\ge L_{N-1}$.
Consequently its reflection with a repeated central pair is log-concave.
\end{lemma}
\begin{proof}
The slopes $d_{K+1},\ldots,d_N$ form a concave sequence. They are at
least the smaller endpoint value, so they are nonnegative. Direct
subtraction gives
\begin{equation}\label{C:eq:propagation}
 D_{k+1}-D_k=w_kd_{k+1}+L_k(w_k-w_{k+1}).
\end{equation}
As long as $w_k\ge0$, the right side is nonnegative. Once $w_k\le0$,
it stays nonpositive and ordinary concavity gives
$D_k=d_kd_{k+1}-L_kw_k\ge0$. This proves the half inequalities.
At the repeated central pair the only additional inequality is equivalent
to $d_N\ge0$.
\end{proof}

\begin{proposition}\label{C:prop:third}
Put $p=\min(n,m)$. If $n,m\ge6$, the second differences are
nonincreasing for
\[
 k\ge p+2\quad(+)\qquad\text{or}\qquad k\ge p+1\quad(-),
\]
through $k=N-2$. More generally, if $\max(n,m)\ge4$, they are
nonincreasing for
\[
 k\ge p+3\quad(+)\qquad\text{or}\qquad k\ge p+2\quad(-).
\]
\end{proposition}
\begin{proof}
By \eqref{C:eq:half}, $w_{k+1}-w_k$ is the coefficient in degree $k+2$
of the numerator divided by $1-z$. Put $t=k+2$ and group the two pairs
as
\[
 B_A(t)=C_{G_A}(t-n)-C_{J_A}(t-2n),\qquad
 B_B(t)=C_{G_B}(t-m)-C_{J_B}(t-2m).
\]
Before the central term begins, the third difference is
$C_B(t)-B_A(t)-B_B(t)$, and $0\le C_B(t)\le9$.

For any packet pair $(G,J)$ and integer $d\ge1$, nonnegativity of $J$
gives
\[
 C_G(x)-C_J(x-d)\ge C_G(x)-C_J(x-1)=:\ell(x).
\]
The eight lists of $\ell(0),\ldots,\ell(7)$ appear in
Section~\ref{C:app:packets}. Every entry is nonnegative, and for $x\ge8$
one has $\ell(x)=9$. Those lists also give
\[
\begin{array}{c|cc}
 &+&-\\\hline
 \ell(x)\ge9&x\ge5&x\ge4\\
 \ell(x)\ge3&x\ge2&\text{not needed}\\
 \ell(x)\ge5&x\ge3&x\ge2\\
 \ell(x)\ge2&\text{not needed}&x\ge1.
\end{array}
\]
If $d\ge6$, inspection only at the unsaturated indices improves the
first row to $x\ge4$ for plus packets and $x\ge3$ for minus packets:
for $x<6$ the subtracted prefix is zero; for subsequent $x$ the $G$
prefix is already large enough, and once it has mass $18$, subtraction
of at most $9$ leaves at least $9$.

Thus, in the stated ranges, the packet associated to the smaller of
$n,m$ contributes at least $9$, and the other contributes at least zero.
This proves nonpositivity before the central contribution.

There are exactly two remaining degrees. For the plus family, at
$t=N-1=n+m+1$ and $t=N=n+m+2$, the extra numerator-prefix contributions
are $1$ and $4$. If $\max(n,m)\ge4$, one packet contributes at least
$9$; the other has argument at least $2$ or $3$, respectively, and
contributes at least $3$ or $5$. Hence the third differences are at most
$9+1-9-3$ and $9+4-9-5$, both negative.
For the minus family these two degrees are $n+m$ and $n+m+1$. One
packet again contributes at least $9$, while the other contributes at least
$2$ and $5$. The corresponding upper bounds are $9+1-9-2$ and
$9+4-9-5$, again negative. This completes the proof.
\end{proof}

\section{Positive initial defects and completion}

\subsection{Both parameters at least six}
Put $p=\min(n,m)\ge6$ and $s=p-6$. The free coefficient
$f_r^\pm(j)=c_B(j)$ agrees with the following cubic once $j\ge4$ in
the plus case and $j\ge2$ in the minus case:
\[
\begin{array}{c|l}
 +,1&(3j^3-15j^2+60j-70)/2\\
 +,2&(3j^3-12j^2+51j-56)/2\\
 -,1&(3j^3-6j^2+21j-10)/2\\
 -,2&(3j^3-3j^2+18j-8)/2.
\end{array}
\]
For earlier arguments use $c_B(j)$ itself. The first four plus defects
are, respectively, $(4,28,109,386)$ and $(10,45,186,544)$; the first
two minus defects are $(9,56)$ and $(16,100)$. All later free defects
are positive: the coefficient lists of four times these quartics, in the
variable $v=j-5\ge0$ for plus and $v=j-3\ge0$ for minus, are
\begin{center}\small
\begin{tabular}{ccl}\toprule
Family&Shift in $j$&Coefficients in descending degree\\\midrule
$+,\,1$ & $5$ & $(27,360,1773,3780,4344)$\\
$+,\,2$ & $5$ & $(27,396,2151,5166,6108)$\\
$-,\,1$ & $3$ & $(27,252,855,1206,960)$\\
$-,\,2$ & $3$ & $(27,288,1125,1944,1680)$\\
\bottomrule\end{tabular}
\end{center}

The first translated subtraction starts at $p+1$ for plus and at $p$ for
minus. Thus only the defects at $p,p+1,p+2$ for plus and at
$p-1,p,p+1$ for minus remain before the tail begins. Neither doubled
translation nor the central contribution meets any of the coefficients
needed in these head triples when $p\ge6$.

For a tuple $(a_4,a_3,a_2,a_1,a_0)$, write
$\Pi(s)=\sum_{i=0}^4 a_i s^i$. The following table gives a coefficientwise
lower bound for $4D_k$ in every parameter chamber, in descending degree.
\begin{center}\small
\begin{tabular}{ccl}\toprule
Family&Degree $k$&Coefficientwise lower bound for $4D_k$\\\midrule
$+,\,1$ & $p$ & $(27,474,3075,8784,10744)$\\
$+,\,1$ & $p+1$ & $(27,600,4857,17340,24512)$\\
$+,\,1$ & $p+2$ & $(27,744,7251,30810,49876)$\\
$+,\,2$ & $p$ & $(27,510,3567,11076,14396)$\\
$+,\,2$ & $p+1$ & $(27,642,5559,21324,32044)$\\
$+,\,2$ & $p+2$ & $(27,798,8211,36624,61672)$\\
$-,\,1$ & $p-1$ & $(27,474,3075,8748,9580)$\\
$-,\,1$ & $p$ & $(27,600,4857,17196,22916)$\\
$-,\,1$ & $p+1$ & $(27,744,7251,30450,46992)$\\
$-,\,2$ & $p-1$ & $(27,510,3567,11040,13220)$\\
$-,\,2$ & $p$ & $(27,642,5559,21144,30136)$\\
$-,\,2$ & $p+1$ & $(27,798,8211,36156,58552)$\\
\bottomrule\end{tabular}
\end{center}
All coefficients and all constant terms are positive. Here is the finite
algebraic calculation underlying the table. The chambers are
\[
 n=m=p;\quad n=p,\ m=p+\delta;\quad m=p,\ n=p+\delta,
\]
where $\delta=1,2$ or $\delta\ge3$. In the second chamber, the relevant
head coefficient is exactly
\[
 L_{p+j}=f_r^\pm(p+j)-c_{G_A}(j)-c_{G_B}(j-\delta),
\]
and in the third the two offsets are interchanged. Negative arguments
contribute zero. The largest offset required is $3$ for plus and $2$
for minus; therefore every $\delta\ge3$ gives the same head correction.
Substitute these expressions into $4D_k$ and then put $p=s+6$.
Section~\ref{C:app:quartics} prints all $84$ resulting quartics. The table
above takes the minimum in each coefficient column for the corresponding
family and degree. This proves a coefficientwise lower bound, and hence a
numerical lower bound for every real $s\ge0$, in each chamber.

The initial slopes of the tails are positive as well. If
$K=p+2$ for plus and $K=p+1$ for minus, the following lower bounds hold:
\[
\begin{array}{c|l}
 +,1&d_{K+1}\ge (9p^2+15p-30)/2\\
 +,2&d_{K+1}\ge (9p^2+21p-32)/2\\
 -,1&d_{K+1}\ge (9p^2+15p-48)/2\\
 -,2&d_{K+1}\ge (9p^2+21p-50)/2.
\end{array}
\]
Indeed, subtract from the free cubic first difference the two largest
possible prefix contributions of $G_A/(1-z)^3$ and $G_B/(1-z)^3$ at
that degree. Those maxima occur when $n=m=p$ because both packets have
nonnegative coefficients. Direct substitution gives the displayed bounds.
They are positive for $p\ge6$.

\subsection{One short parameter}
Suppose $1\le p=\min(n,m)\le5$ and the other parameter is at least
$p+4$. Put $K=p+3$ for plus and $K=p+2$ for minus. Through $K+1$,
all packets translated by the larger parameter vanish; so do the central
terms. Thus the required head is independent of the larger parameter.
The next table gives its minimum defect through $K$ and its initial tail
step $d_{K+1}$, calculated by \eqref{C:eq:half}. Each pair of columns is
$(\min D,d_{K+1})$.
\begin{center}\small
\begin{tabular}{cc rr rr rr rr}\toprule
$p$&Short parameter&\multicolumn{2}{c}{$+,1$}&\multicolumn{2}{c}{$+,2$}&\multicolumn{2}{c}{$-,1$}&\multicolumn{2}{c}{$-,2$}\\\midrule
$1$ & $n=p$ & 6 & 14 & 12 & 22 & 2 & 10 & 5 & 17\\
$1$ & $m=p$ & 5 & 23 & 11 & 24 & 5 & 18 & 11 & 20\\
$2$ & $n=p$ & 4 & 40 & 10 & 52 & 11 & 34 & 18 & 45\\
$2$ & $m=p$ & 4 & 50 & 10 & 53 & 10 & 43 & 17 & 47\\
$3$ & $n=p$ & 4 & 78 & 10 & 93 & 9 & 70 & 16 & 85\\
$3$ & $m=p$ & 4 & 88 & 10 & 94 & 9 & 80 & 16 & 86\\
$4$ & $n=p$ & 4 & 126 & 10 & 144 & 9 & 117 & 16 & 135\\
$4$ & $m=p$ & 4 & 136 & 10 & 145 & 9 & 127 & 16 & 136\\
$5$ & $n=p$ & 4 & 183 & 10 & 204 & 9 & 174 & 16 & 195\\
$5$ & $m=p$ & 4 & 193 & 10 & 205 & 9 & 184 & 16 & 196\\
\bottomrule\end{tabular}
\end{center}
All these entries are positive. The remaining parameters have
$1\le p\le5$ and the other parameter in $\{p,p+1,p+2,p+3\}$.
There are $35$ ordered pairs, each with four signs/residues.
Section~\ref{C:app:bases} gives their minimum internal defects and final
half-steps. Each is obtained from the same explicit binomial coefficient
formula, with no unbounded parameter remaining.

\begin{proof}[Proof of Theorem~\ref{C:thm:main}]
If $n,m\ge6$, the free-head calculation and the quartic calculations
prove all defects through $K=p+2$ for plus and $K=p+1$ for minus,
and the displayed initial tail slopes are positive.
If $p\le5$ and the larger parameter is at least $p+4$, the short-head
table proves the corresponding assertions with $K=p+3$ or $p+2$.
In both cases Proposition~\ref{C:prop:third} gives decreasing second
differences on the rest of the rising half, and Lemma~\ref{C:lem:central}
gives the final half-step. Apply Lemma~\ref{C:lem:propagation} and then
reflection. Section~\ref{C:app:bases} covers every remaining parameter.
\end{proof}

\section{Packet coefficient lists}\label{C:app:packets}
In this table a list starts in degree zero. All omitted later coefficients
are zero.
\begin{center}\small
\begin{longtable}{ccl}\toprule
Family&Packet&Coefficient list\\\midrule\endhead
$+,\,1$ & $B$ & $(1,0,2,1,2,1,1,1)$\\
$+,\,1$ & $G_A$ & $(0,2,2,3,4,3,2,2)$\\
$+,\,1$ & $G_B$ & $(0,1,2,3,4,4,3,1)$\\
$+,\,1$ & $J_A$ & $(0,0,1,2,2,2,1,1)$\\
$+,\,1$ & $J_B$ & $(0,0,1,1,2,2,2,1)$\\
$+,\,2$ & $B$ & $(1,1,1,2,1,2,0,1)$\\
$+,\,2$ & $G_A$ & $(0,2,2,3,4,3,2,2)$\\
$+,\,2$ & $G_B$ & $(0,1,3,4,4,3,2,1)$\\
$+,\,2$ & $J_A$ & $(0,0,1,1,2,2,2,1)$\\
$+,\,2$ & $J_B$ & $(0,0,1,2,2,2,1,1)$\\
$-,\,1$ & $B$ & $(1,1,2,2,2,1)$\\
$-,\,1$ & $G_A$ & $(2,2,3,4,3,2,2)$\\
$-,\,1$ & $G_B$ & $(1,2,3,4,4,3,1)$\\
$-,\,1$ & $J_A$ & $(1,1,1,2,1,2,0,1)$\\
$-,\,1$ & $J_B$ & $(1,0,2,1,2,1,1,1)$\\
$-,\,2$ & $B$ & $(1,2,2,2,1,1)$\\
$-,\,2$ & $G_A$ & $(2,2,3,4,3,2,2)$\\
$-,\,2$ & $G_B$ & $(1,3,4,4,3,2,1)$\\
$-,\,2$ & $J_A$ & $(1,0,2,1,2,1,1,1)$\\
$-,\,2$ & $J_B$ & $(1,1,1,2,1,2,0,1)$\\
\bottomrule\end{longtable}
\end{center}
For each pair $(G_A,J_A)$ and $(G_B,J_B)$, the next table gives
$\ell(x)=C_G(x)-C_J(x-1)$ at $x=0,\ldots,7$. For $x<0$ it is zero;
for $x\ge8$ it equals nine.
\begin{center}\small
\begin{tabular}{ccl}\toprule
Family&Pair&$\ell(0),\ldots,\ell(7)$\\\midrule
$+,\,1$ & $A$ & $(0,2,4,6,8,9,9,10)$\\
$+,\,1$ & $B$ & $(0,1,3,5,8,10,11,10)$\\
$+,\,2$ & $A$ & $(0,2,4,6,9,10,10,10)$\\
$+,\,2$ & $B$ & $(0,1,4,7,9,10,10,10)$\\
$-,\,1$ & $A$ & $(2,3,5,8,9,10,10,10)$\\
$-,\,1$ & $B$ & $(1,2,5,7,10,11,11,10)$\\
$-,\,2$ & $A$ & $(2,3,6,8,10,10,11,10)$\\
$-,\,2$ & $B$ & $(1,3,6,9,10,11,10,10)$\\
\bottomrule\end{tabular}
\end{center}

\section{The short seed interval layers}\label{C:app:layers}
Each row lists all superlevel intervals. In every row each $A$ midpoint
is no smaller than every $U$ midpoint.
\begin{center}\small
\begin{tabular}{ccll}\toprule
Family&$n$&Layers of $A$&Layers of $U$\\\midrule
$+,\,1$ & $1$ & $[1,3]$ & $[0,3]$\\
$+,\,1$ & $2$ & $[2,4]$ & $[0,4],[2,3]$\\
$+,\,1$ & $3$ & $[3,5]$ & $[0,5],[2,4],[3,3]$\\
$+,\,2$ & $1$ & $[1,3]$ & $[0,3],[1,1]$\\
$+,\,2$ & $2$ & $[2,4]$ & $[0,4],[1,3]$\\
$+,\,2$ & $3$ & $[3,5]$ & $[0,5],[1,4],[3,3]$\\
$-,\,1$ & $1$ & $[3,3]$ & $[0,3]$\\
$-,\,1$ & $2$ & $[1,4]$ & $[0,4],[1,2]$\\
$-,\,1$ & $3$ & $[1,5],[2,3]$ & $[0,5],[1,3],[2,2]$\\
$-,\,2$ & $1$ & $[2,3]$ & $[0,3],[2,2]$\\
$-,\,2$ & $2$ & $[1,4],[3,3]$ & $[0,4],[1,3]$\\
$-,\,2$ & $3$ & $[1,5],[2,4]$ & $[0,5],[1,4],[2,2]$\\
\bottomrule\end{tabular}
\end{center}

\section{All head quartics}\label{C:app:quartics}
A tuple $(a_4,a_3,a_2,a_1,a_0)$ denotes
$4D_k=a_4s^4+a_3s^3+a_2s^2+a_1s+a_0$, with $s=p-6\ge0$.
These are exact quartics, not fitted values.
\begin{center}\scriptsize
\begin{longtable}{ccll}\toprule
Family&Parameter chamber&Degree&Coefficients of $4D_k$\\\midrule\endhead
$+,\,1$ & $n=m=p$ & $p$ & $(27,486,3195,9324,11664)$\\
$+,\,1$ & $n=m=p$ & $p+1$ & $(27,636,5253,19092,27352)$\\
$+,\,1$ & $n=m=p$ & $p+2$ & $(27,822,8211,35088,56472)$\\
$+,\,1$ & $n=p,\ m=p+1$ & $p$ & $(27,480,3135,9054,11204)$\\
$+,\,1$ & $n=p,\ m=p+1$ & $p+1$ & $(27,618,5055,18216,25928)$\\
$+,\,1$ & $n=p,\ m=p+1$ & $p+2$ & $(27,786,7779,33240,53776)$\\
$+,\,1$ & $n=p,\ m=p+2$ & $p$ & $(27,480,3135,9054,11204)$\\
$+,\,1$ & $n=p,\ m=p+2$ & $p+1$ & $(27,612,4977,17808,25132)$\\
$+,\,1$ & $n=p,\ m=p+2$ & $p+2$ & $(27,768,7527,31914,51236)$\\
$+,\,1$ & $n=p,\ m\ge p+3$ & $p$ & $(27,480,3135,9054,11204)$\\
$+,\,1$ & $n=p,\ m\ge p+3$ & $p+1$ & $(27,612,4977,17808,25132)$\\
$+,\,1$ & $n=p,\ m\ge p+3$ & $p+2$ & $(27,762,7431,31332,49956)$\\
$+,\,1$ & $m=p,\ n=p+1$ & $p$ & $(27,474,3075,8784,10744)$\\
$+,\,1$ & $m=p,\ n=p+1$ & $p+1$ & $(27,612,5013,18156,26104)$\\
$+,\,1$ & $m=p,\ n=p+1$ & $p+2$ & $(27,780,7755,33462,54940)$\\
$+,\,1$ & $m=p,\ n=p+2$ & $p$ & $(27,474,3075,8784,10744)$\\
$+,\,1$ & $m=p,\ n=p+2$ & $p+1$ & $(27,600,4857,17340,24512)$\\
$+,\,1$ & $m=p,\ n=p+2$ & $p+2$ & $(27,756,7443,31974,52444)$\\
$+,\,1$ & $m=p,\ n\ge p+3$ & $p$ & $(27,474,3075,8784,10744)$\\
$+,\,1$ & $m=p,\ n\ge p+3$ & $p+1$ & $(27,600,4857,17340,24512)$\\
$+,\,1$ & $m=p,\ n\ge p+3$ & $p+2$ & $(27,744,7251,30810,49876)$\\
$+,\,2$ & $n=m=p$ & $p$ & $(27,522,3699,11700,15492)$\\
$+,\,2$ & $n=m=p$ & $p+1$ & $(27,678,5991,23352,35620)$\\
$+,\,2$ & $n=m=p$ & $p+2$ & $(27,876,9321,42420,72452)$\\
$+,\,2$ & $n=p,\ m=p+1$ & $p$ & $(27,516,3633,11388,14944)$\\
$+,\,2$ & $n=p,\ m=p+1$ & $p+1$ & $(27,654,5691,21876,32976)$\\
$+,\,2$ & $n=p,\ m=p+1$ & $p+2$ & $(27,828,8685,39420,67608)$\\
$+,\,2$ & $n=p,\ m=p+2$ & $p$ & $(27,516,3633,11388,14944)$\\
$+,\,2$ & $n=p,\ m=p+2$ & $p+1$ & $(27,648,5607,21414,32044)$\\
$+,\,2$ & $n=p,\ m=p+2$ & $p+2$ & $(27,804,8313,37272,63148)$\\
$+,\,2$ & $n=p,\ m\ge p+3$ & $p$ & $(27,516,3633,11388,14944)$\\
$+,\,2$ & $n=p,\ m\ge p+3$ & $p+1$ & $(27,648,5607,21414,32044)$\\
$+,\,2$ & $n=p,\ m\ge p+3$ & $p+2$ & $(27,798,8211,36624,61672)$\\
$+,\,2$ & $m=p,\ n=p+1$ & $p$ & $(27,510,3567,11076,14396)$\\
$+,\,2$ & $m=p,\ n=p+1$ & $p+1$ & $(27,654,5727,22248,34068)$\\
$+,\,2$ & $m=p,\ n=p+1$ & $p+2$ & $(27,834,8823,40476,70364)$\\
$+,\,2$ & $m=p,\ n=p+2$ & $p$ & $(27,510,3567,11076,14396)$\\
$+,\,2$ & $m=p,\ n=p+2$ & $p+1$ & $(27,642,5559,21324,32204)$\\
$+,\,2$ & $m=p,\ n=p+2$ & $p+2$ & $(27,810,8487,38772,67356)$\\
$+,\,2$ & $m=p,\ n\ge p+3$ & $p$ & $(27,510,3567,11076,14396)$\\
$+,\,2$ & $m=p,\ n\ge p+3$ & $p+1$ & $(27,642,5559,21324,32204)$\\
$+,\,2$ & $m=p,\ n\ge p+3$ & $p+2$ & $(27,798,8283,37476,64396)$\\
$-,\,1$ & $n=m=p$ & $p-1$ & $(27,486,3195,9216,10260)$\\
$-,\,1$ & $n=m=p$ & $p$ & $(27,636,5253,18732,24964)$\\
$-,\,1$ & $n=m=p$ & $p+1$ & $(27,822,8211,34260,51948)$\\
$-,\,1$ & $n=p,\ m=p+1$ & $p-1$ & $(27,480,3135,8982,9920)$\\
$-,\,1$ & $n=p,\ m=p+1$ & $p$ & $(27,618,5055,17964,23936)$\\
$-,\,1$ & $n=p,\ m=p+1$ & $p+1$ & $(27,786,7779,32628,50116)$\\
$-,\,1$ & $n=p,\ m=p+2$ & $p-1$ & $(27,480,3135,8982,9920)$\\
$-,\,1$ & $n=p,\ m=p+2$ & $p$ & $(27,612,4977,17592,23296)$\\
$-,\,1$ & $n=p,\ m=p+2$ & $p+1$ & $(27,768,7527,31410,48080)$\\
$-,\,1$ & $n=p,\ m\ge p+3$ & $p-1$ & $(27,480,3135,8982,9920)$\\
$-,\,1$ & $n=p,\ m\ge p+3$ & $p$ & $(27,612,4977,17592,23296)$\\
$-,\,1$ & $n=p,\ m\ge p+3$ & $p+1$ & $(27,762,7431,30864,46992)$\\
$-,\,1$ & $m=p,\ n=p+1$ & $p-1$ & $(27,474,3075,8748,9580)$\\
$-,\,1$ & $m=p,\ n=p+1$ & $p$ & $(27,612,5013,17940,24196)$\\
$-,\,1$ & $m=p,\ n=p+1$ & $p+1$ & $(27,780,7755,32886,51328)$\\
$-,\,1$ & $m=p,\ n=p+2$ & $p-1$ & $(27,474,3075,8748,9580)$\\
$-,\,1$ & $m=p,\ n=p+2$ & $p$ & $(27,600,4857,17196,22916)$\\
$-,\,1$ & $m=p,\ n=p+2$ & $p+1$ & $(27,756,7443,31542,49456)$\\
$-,\,1$ & $m=p,\ n\ge p+3$ & $p-1$ & $(27,474,3075,8748,9580)$\\
$-,\,1$ & $m=p,\ n\ge p+3$ & $p$ & $(27,600,4857,17196,22916)$\\
$-,\,1$ & $m=p,\ n\ge p+3$ & $p+1$ & $(27,744,7251,30450,47272)$\\
$-,\,2$ & $n=m=p$ & $p-1$ & $(27,522,3699,11592,14052)$\\
$-,\,2$ & $n=m=p$ & $p$ & $(27,678,5991,22956,32944)$\\
$-,\,2$ & $n=m=p$ & $p+1$ & $(27,876,9321,41484,67112)$\\
$-,\,2$ & $n=p,\ m=p+1$ & $p-1$ & $(27,516,3633,11316,13636)$\\
$-,\,2$ & $n=p,\ m=p+1$ & $p$ & $(27,654,5691,21624,30900)$\\
$-,\,2$ & $n=p,\ m=p+1$ & $p+1$ & $(27,828,8685,38772,63540)$\\
$-,\,2$ & $n=p,\ m=p+2$ & $p-1$ & $(27,516,3633,11316,13636)$\\
$-,\,2$ & $n=p,\ m=p+2$ & $p$ & $(27,648,5607,21198,30136)$\\
$-,\,2$ & $n=p,\ m=p+2$ & $p+1$ & $(27,804,8313,36768,59824)$\\
$-,\,2$ & $n=p,\ m\ge p+3$ & $p-1$ & $(27,516,3633,11316,13636)$\\
$-,\,2$ & $n=p,\ m\ge p+3$ & $p$ & $(27,648,5607,21198,30136)$\\
$-,\,2$ & $n=p,\ m\ge p+3$ & $p+1$ & $(27,798,8211,36156,58552)$\\
$-,\,2$ & $m=p,\ n=p+1$ & $p-1$ & $(27,510,3567,11040,13220)$\\
$-,\,2$ & $m=p,\ n=p+1$ & $p$ & $(27,654,5727,21996,31920)$\\
$-,\,2$ & $m=p,\ n=p+1$ & $p+1$ & $(27,834,8823,39792,66020)$\\
$-,\,2$ & $m=p,\ n=p+2$ & $p-1$ & $(27,510,3567,11040,13220)$\\
$-,\,2$ & $m=p,\ n=p+2$ & $p$ & $(27,642,5559,21144,30392)$\\
$-,\,2$ & $m=p,\ n=p+2$ & $p+1$ & $(27,810,8487,38232,63684)$\\
$-,\,2$ & $m=p,\ n\ge p+3$ & $p-1$ & $(27,510,3567,11040,13220)$\\
$-,\,2$ & $m=p,\ n\ge p+3$ & $p$ & $(27,642,5559,21144,30392)$\\
$-,\,2$ & $m=p,\ n\ge p+3$ & $p+1$ & $(27,798,8283,37008,61132)$\\
\bottomrule\end{longtable}
\end{center}

\section{Finite base calculations}\label{C:app:bases}
Each pair of columns gives the minimum internal log-concavity defect and
the final half-step of the complete half-sequence. The binomial formula
\eqref{C:eq:half} determines every entry without any recurrence guess.
The accompanying data file lists all $140$ profiles; the displayed binomial formula and the table below specify the finite calculation without requiring that file.
\begin{center}\small
\begin{longtable}{c rr rr rr rr}\toprule
$(n,m)$&\multicolumn{2}{c}{$+,1$}&\multicolumn{2}{c}{$+,2$}&\multicolumn{2}{c}{$-,1$}&\multicolumn{2}{c}{$-,2$}\\\midrule\endhead
$(1,1)$ & 7 & 4 & 13 & 6 & 1 & 1 & 3 & 3\\
$(1,2)$ & 6 & 5 & 12 & 9 & 3 & 2 & 6 & 5\\
$(1,3)$ & 6 & 6 & 12 & 11 & 2 & 4 & 5 & 8\\
$(1,4)$ & 6 & 7 & 12 & 12 & 2 & 6 & 5 & 11\\
$(2,1)$ & 5 & 8 & 11 & 9 & 7 & 4 & 13 & 6\\
$(3,1)$ & 5 & 11 & 11 & 11 & 5 & 8 & 11 & 9\\
$(4,1)$ & 5 & 12 & 11 & 12 & 5 & 11 & 11 & 11\\
$(2,2)$ & 4 & 12 & 10 & 15 & 12 & 8 & 19 & 11\\
$(2,3)$ & 4 & 16 & 10 & 21 & 11 & 13 & 18 & 17\\
$(2,4)$ & 4 & 19 & 10 & 25 & 11 & 17 & 18 & 22\\
$(2,5)$ & 4 & 20 & 10 & 26 & 11 & 19 & 18 & 25\\
$(3,2)$ & 4 & 19 & 10 & 20 & 10 & 15 & 17 & 17\\
$(4,2)$ & 4 & 23 & 10 & 23 & 10 & 20 & 17 & 21\\
$(5,2)$ & 4 & 24 & 10 & 24 & 10 & 23 & 17 & 23\\
$(3,3)$ & 4 & 27 & 10 & 30 & 9 & 23 & 16 & 26\\
$(3,4)$ & 4 & 33 & 10 & 38 & 9 & 30 & 16 & 34\\
$(3,5)$ & 4 & 36 & 10 & 42 & 9 & 34 & 16 & 39\\
$(3,6)$ & 4 & 37 & 10 & 43 & 9 & 36 & 16 & 42\\
$(4,3)$ & 4 & 35 & 10 & 36 & 9 & 31 & 16 & 33\\
$(5,3)$ & 4 & 39 & 10 & 39 & 9 & 36 & 16 & 37\\
$(6,3)$ & 4 & 40 & 10 & 40 & 9 & 39 & 16 & 39\\
$(4,4)$ & 4 & 45 & 10 & 48 & 9 & 41 & 16 & 44\\
$(4,5)$ & 4 & 51 & 10 & 56 & 9 & 48 & 16 & 52\\
$(4,6)$ & 4 & 54 & 10 & 60 & 9 & 52 & 16 & 57\\
$(4,7)$ & 4 & 55 & 10 & 61 & 9 & 54 & 16 & 60\\
$(5,4)$ & 4 & 53 & 10 & 54 & 9 & 49 & 16 & 51\\
$(6,4)$ & 4 & 57 & 10 & 57 & 9 & 54 & 16 & 55\\
$(7,4)$ & 4 & 58 & 10 & 58 & 9 & 57 & 16 & 57\\
$(5,5)$ & 4 & 63 & 10 & 66 & 9 & 59 & 16 & 62\\
$(5,6)$ & 4 & 69 & 10 & 74 & 9 & 66 & 16 & 70\\
$(5,7)$ & 4 & 72 & 10 & 78 & 9 & 70 & 16 & 75\\
$(5,8)$ & 4 & 73 & 10 & 79 & 9 & 72 & 16 & 78\\
$(6,5)$ & 4 & 71 & 10 & 72 & 9 & 67 & 16 & 69\\
$(7,5)$ & 4 & 75 & 10 & 75 & 9 & 72 & 16 & 73\\
$(8,5)$ & 4 & 76 & 10 & 76 & 9 & 75 & 16 & 75\\
\bottomrule\end{longtable}
\end{center}

\section{Euclidean histogram input}\label{C:app:lift}
Let the seed $W(u,b-u)$ have potential $v_s$, $1\le s\le b$, and
let $\tau$ be its top reflection, with blocks of lengths $u,b-u$.
For $a=qb+u$, $q\ge1$, put $Q_s=q$ in the first seed block and
$Q_s=q-1$ in the second. The parent values are
\[
 \psi(s+jb)=\min\{v_s+2j,\ v_{\tau(s)}+1+2(Q_s-j)\},
 \qquad 0\le j\le Q_s,
\]
and $\psi(a+t)=v_{b+1-t}-1$ on the second parent block.
For a two-element seed orbit $s<t=\tau(s)$ with values $v,v+1$ and
common $Q$, the two chains are
\[
 v+\min(2j,2Q-2j+2),\qquad
 v+1+\min(2j,2Q-2j).
\]
The top mate of $s+jb$ is $t+(Q-j)b$. If $2j\le Q$, their values
are $v+2j,v+2j+1$ in increasing vertex order; otherwise the order
reverses and their values are $v+2Q-2j+2,v+2Q-2j+1$.
For $1\le j\le Q$, the bottom mate is $t+(Q+1-j)b$. If
$2j\le Q+1$, the values in increasing vertex order are $v+2j,v+2j-1$;
otherwise the order reverses and the values are
$v+2Q-2j+2,v+2Q-2j+3$. Thus every internal top and bottom edge has
exactly the required potential difference.
A fixed seed residue has the single chain
$v_s+\min(2j,2Q_s-2j+1)$. The same checks use indices $j,Q_s-j$
and $j,Q_s+1-j$; equal indices denote fixed reflections, not edges.
Boundary bottom pairs join $v_s$ to $v_s-1$, and reversed second-block
top pairs are seed bottom pairs. This verifies all parent edges.

The union of a paired orbit's values has polynomial
$z^v(1+z)I_{Q+1}$: its two extreme heights occur once, and the intervening
heights occur twice. A fixed orbit contributes $z^{v_s}I_{Q_s+1}$,
because its alternating parity branches give each height once.
Summing the orbits proves
$H_q=AI_{q+1}+CI_q=(A+C)I_q+z^qA$. The second block has
histogram $G=z^{-1}(A+C)$.

For $q=0$, on the first block of size $u$ assign the first seed-block
values, and on the second block of size $b$ assign the reversed seed
values minus one. First-block top edges stay seed top edges; second-block
top edges become seed bottom edges. The bottom edges incident to the first
block join a seed value to that value minus one. For a bottom edge lying
entirely in the second block, its reversed seed indices lie in the second
seed top block and are top mates there. Their seed top relation is exactly
the required parent bottom relation. This proves \eqref{C:eq:lift} at zero
as an actual meander construction. It also covers the $m=1$ minus cases.